\documentclass[10pt]{article}
\usepackage{geometry}

\usepackage{comment,hyperref,slashed,tensor}

\usepackage{amsmath}
\usepackage[T1]{fontenc}
\usepackage[utf8]{inputenc}
\DeclareMathOperator{\supp}{supp}

\usepackage{ mathrsfs }
\usepackage{amsmath, amsfonts, mathtools, amsthm, amssymb}

\mathtoolsset{showonlyrefs}
\usepackage{todonotes, tikz}
\usepackage{hyperref,  color}

\newtheorem{theorem}{Theorem}[section]
\newtheorem{lemma}[theorem]{Lemma}
\newtheorem{proposition}[theorem]{Proposition}

\newtheorem{cor}[theorem]{Corollary}

\newtheorem{definition}[theorem]{Definition}

\newtheorem{question}[theorem]{Question}

\newtheorem{remark}[theorem]{Remark}
\newcommand{\be}{\begin{equation}}
\newcommand{\ee}{\end{equation}}

\def\XXint#1#2#3{{\setbox0=\hbox{$#1{#2#3}{\int}$ }
\vcenter{\hbox{$#2#3$ }}\kern-.6\wd0}}

\title{On the flexibility of 2D Euler steady states}
\author{Tarek M. Elgindi and Yupei Huang}
\date{\today}

\begin{document}

\maketitle

\begin{abstract}
We consider steady states of the incompressible Euler equation on two-dimensional domains. For non-radial analytic steady states on bounded simply connected domains, it was shown previously that there must be a global functional relationship between the stream function and the vorticity. We show that this does not extend to smooth functions, even under further structural assumptions such as the Morse condition or Arnold's stability criterion. More precisely, we show that a broad class of steady states with multiple critical points can be perturbed to smooth steady states for which the vorticity is not a single-valued function of the stream function. We also establish an analogous flexibility result near the cellular flow on the flat torus, which is a degenerate case. As a consequence of our constructions, there are "branches" of smooth steady states that are isolated from analytic ones. In some cases, the resulting isolated branches can even consist entirely of linearly stable steady states.  
\end{abstract}

\section{Introduction}

Steady solutions play a central role in the dynamics of evolution equations. 
To understand the long-time behavior of unsteady solutions, one first needs a 
good understanding of the structure of the steady ones. For the two-dimensional 
incompressible Euler equations, there are large classes of steady solutions 
exhibiting a wide variety of structures. Naturally, there is a long and extensive 
literature on their stability 
\cite{arnold1965conditions,lin2004some,arnold1966priori,lin2004nonlinear,lin2003instability,cao2026instability,zhao2024inviscid}, 
flexibility 
\cite{choffrut2012local,coti2023stationary,elgindi2022regular,lin2011inviscid,constantin2021flexibility,enciso2024smooth}, 
and rigidity 
\cite{ruiz2023symmetry,hamel2017shear,gomez2021symmetry,hamel2019liouville,gui2024classification,ElgindiHuangSaidXie_ClassificationSteadyEulerFlows_DMJ,drivas2024geometric}.    

\subsection{Flexibility, Rigidity, and Stability}
Steady 2D Euler solutions on a simply connected domain $\Omega$ are given by sufficiently smooth functions $\psi:\bar\Omega\rightarrow\mathbb{R}$ that are constant on $\partial\Omega$ for which:
\begin{equation}\label{SEE1} \nabla^\perp\psi\cdot\nabla\Delta\psi=0. \end{equation} Here, $\psi$ represents the stream function, $\nabla^\perp$ is simply the rotated gradient, and $\Delta\psi$ is often denoted by $\omega,$ the vorticity, while the velocity field $\nabla^\perp\psi$ is denoted by $u$. 
 \eqref{SEE1} has a large class of solutions on any domain. This is intricately related to the fact that the equations have infinitely many conserved quantities \cite{marchioro2012mathematical}. On general simply connected domains, the known conserved quantities are\footnote{On multiply connected domains, we would also include the circulations on the boundary. The Kelvin circulation theorem gives further conserved quantities, though these are not local in time in the sense that they may depend on the time evolution of the velocity field (rather than the velocity at a particular time).}:
\begin{equation}\label{CLs}E=\int |u|^2,\qquad J_{f}=\int f(\omega).\end{equation} One technique to construct steady solutions is to extremize $E$ when one of the functionals $J_f$ is fixed. In the case that $f$ is convex, this procedure yields a Lyapunov stable steady state. This immediately yields an infinite-dimensional family of steady states on any reasonable domain, with each stream function solving a semilinear elliptic equation:
\begin{equation}\label{SEE2}\Delta\psi = F(\psi),\end{equation}
with $F=f'$. It was proven recently in \cite{ElgindiHuangSaidXie_ClassificationSteadyEulerFlows_DMJ} that, in the analytic class, \eqref{SEE1} and \eqref{SEE2} are equivalent\footnote{More precisely, a non-radial analytic function $\psi$ satisfies \eqref{SEE1} if and only if there exists an $F$ for which \eqref{SEE2} holds. }. One could say that this implies a correspondence between the conservation laws \eqref{CLs} and the steady states, at least in the analytic class. Outside of the analytic class, the equivalence of \eqref{SEE1} and \eqref{SEE2} is false due to the existence of compactly supported radial solutions (see also \cite{GomezSerranoCompact,enciso2024smooth} for further nontrivial compactly supported solutions); however, such counterexamples are non-generic and likely to be dynamically unstable. Given the equivalence of \eqref{SEE1} and \eqref{SEE2} in the analytic class, it is natural to wonder how rare such counterexamples can be. For example, can they be ruled out by assuming that the steady state is Morse or even "stable"? We will show that this is not the case, and in fact, there are open sets in function space in which all steady states to the Euler equation are Morse, satisfy Arnold's stability criterion, but \emph{do not} satisfy \eqref{SEE2}.
\subsubsection{Arnold's Stability Criterion}
As we mentioned above, the variational principle using convex conserved quantities \eqref{CLs} to construct steady states, yields steady states that are stable in the $L^2$ norm of the vorticity. This is sometimes called Kelvin's variational principle \cite{Kelvin,Burton}, but was established rigorously in greater generality by Arnold \cite{arnold1965conditions}. Through the lens of his geometric formalism, Arnold saw that this could be generalized to all steady states in the following way. First, observe that \emph{$\omega_*$ is steady if and only if it is a critical point of the energy among the coadjoint orbit of area-preserving diffeomorphisms.} Second, observe that when $\psi_*$ is a sufficiently smooth Morse function, the ratio $\frac{\nabla\omega_*}{\nabla\psi_*}$ can be computed globally as (at least) a bounded function. The second variation of the energy can then be formally computed as:
\[B_*(\omega)=\int |\nabla \Delta^{-1}\omega|^2+\frac{\nabla\psi_*}{\nabla\omega_*}|\omega|^2,\] where $\omega$ is the variation of vorticity. 
This motivates Arnold's Stability Criterion:
\begin{definition}
A steady solution $\omega_*$ of the Euler equation is said to satisfy Arnold's
stability criterion if the quadratic form $B_*$ is coercive with a definite sign.
\end{definition}
\noindent The positivity of $B_*$ is implied by:
\begin{equation}\label{ASC} 0<c\leq \frac{\nabla\omega_*}{\nabla\psi_*} \leq C<\infty\end{equation} for some real numbers $c,C.$ \footnote{In the case where the stream function $\psi_{*}$ solving a semilinear elliptic equation $\Delta \psi_{*}=F(\psi_{*})$, \eqref{ASC} is equivalent to $F^{'}(\psi_{*})$ being a positive function.}

The definition is motivated by the finite-dimensional fact:
\begin{lemma}
Consider a dynamical system on $\mathbb{R}^n$: \[\dot{x}(t)=N(x(t))\]  defined by a continuously differentiable $N:\mathbb{R}^n\rightarrow\mathbb{R}^n$ with a conserved quantity $E:\mathbb{R}^n\rightarrow\mathbb{R}$. If $x_*$ is a strict local extremizer of $E,$ then $x_*$ is a Lyapunov stable fixed point for the dynamical system. 
\end{lemma}
While infinite dimensional analogues of this may be false \cite{nadirashvili1991wandering,wolansky1998nonlinear,ball1984quasiconvexity}, for a steady state satisfying Arnold's criterion and \emph{also} satisfying \eqref{SEE2} globally, nonlinear stability in $L^2$ of vorticity is known to hold \cite{arnold1998topological,sverak2011course}. Since we must contend with many completely reasonable steady states that do not satisfy \eqref{SEE2}, we thus close this discussion with a question (see also the related Question \ref{Q2}). 
\begin{question}\label{Q1}
Does there exist a steady state $\psi_*$ satisfying \eqref{ASC}  that is nonlinearly unstable in $L^2$ of vorticity?
\end{question}
\noindent We remark that either answer to this question would be very interesting: a positive answer might indicate a "return to symmetry" phenomenon for nonlinear solutions, while a negative answer would indicate that there are forces beyond the basic conserved quantities \eqref{CLs} that govern nonlinear stability (see \cite{KI} for some other possible quantities).

\subsection{Main Results}
We now state the main theorems. 
\begin{theorem}\label{thm: main theorem}
Let $\Omega\subset \mathbb{R}^2$ be bounded and simply connected, $k\geq 6$ be a positive integer and  $\psi_0$ be a $C^{k}$ Morse function vanishing on the boundary for which $\Delta\psi_0=F(\psi_0)$ for some $F$. Assume that the Schr\"odinger operator $\Delta-F'(\psi_0):H^2_0\rightarrow L^2$ is invertible. 
\begin{itemize}
    \item   If $\psi_0$ has more than one critical point in $\Omega,$ then in any $C^{\lfloor\frac{k}{2}-1\rfloor}$ neighborhood of $\psi_0$,  there exists a smooth  steady Euler solution $\tilde\psi$ such that there is no function $\tilde F$ for which $\Delta\tilde\psi=\tilde F(\tilde\psi)$ in $\Omega.$ 
    \item If we further assume that $F'(\psi_0)>0$ uniformly in $\bar\Omega,$ the new steady state $\tilde\psi$ can be chosen to satisfy Arnold's stability criterion \eqref{ASC} and is, thus, linearly stable.  
\end{itemize}
\end{theorem}

\begin{remark}
On a family of Neumann ovals,
the unique solution $\psi_0$ to 
\[\Delta\psi_0=1+\lambda\psi_0\] 
vanishing on the boundary satisfies the conditions of both parts of the Theorem \ref{thm: main theorem} for a fixed small $\lambda>0$ (see Proposition \ref{thm:construction}).
By applying Theorem \ref{thm: main theorem} to such a $\psi_0$, there is a smooth linearly stable steady state whose stream function $\tilde{\psi}$ fails to solve a semilinear elliptic equation globally. As $\tilde{\psi}$ is close to $\psi_0$ in $C^2$ and $\psi_0$ is a Morse function, $\tilde{\psi}$ is also a Morse function.
\end{remark}
\begin{figure}[h!]\label{NO}
 \caption{A plot of the Neumann oval and $\psi_0$. Note that the red points are minima of $\psi_0$ points while the green point at the center is a saddle point.}
    \centering
\includegraphics[width=0.5\linewidth]{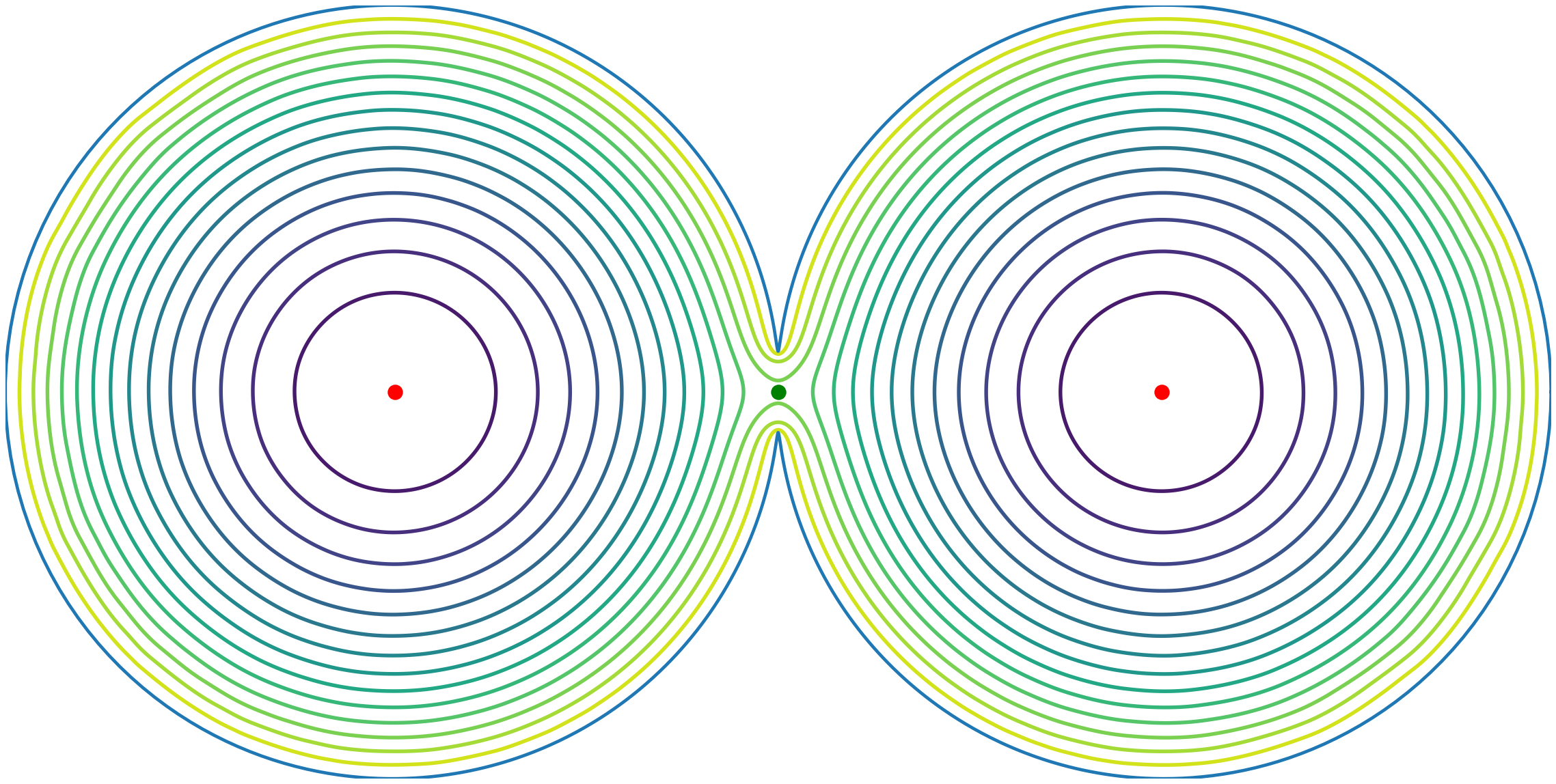}
\end{figure} 
We state a corollary of Theorem \ref{thm: main theorem} and the main theorem in \cite{ElgindiHuangSaidXie_ClassificationSteadyEulerFlows_DMJ} to emphasize the point that these smooth steady states are isolated from analytic ones: 
\begin{cor}\label{cor:isolated}
Let $\Omega$ be a simply connected domain for which there exists a $\psi_0$ satisfying the conditions of Theorem \ref{thm: main theorem}. If we denote by $S_k$ the set of $C^k$ steady states on $\bar\Omega$ (endowed with the $C^k$ topology), we have that the set of analytic steady states $C^\omega\cap S_k$ is not dense in $S_k$ for any $k\geq 3.$ 
\end{cor}
Taken together, these results suggest a different picture of the space of
smooth steady Euler flows from the one in analytic steady Euler flows. In the category of non-radial analytic functions, the steady equation enforces a global functional relation
between the vorticity and the stream function. Thus analytic steady states are
closely tied to semilinear elliptic equations and to the classical
energy--Casimir variational framework. Our results show that this picture is
not robust in the smooth category. Even near steady states satisfying
Arnold's stability criterion and Morse assumptions, there are smooth steady states for which the
vorticity is not a single-valued function of the stream function. In particular,
Corollary~\ref{cor:isolated} shows that these steady states are not
merely isolated examples: analytic steady states fail to be dense in the smooth
steady-state set. This reflects that the
topological structure of the smooth steady states is substantially richer than that of
the analytic steady states.\par 
Theorem~\ref{thm: main theorem} treats a nondegenerate regime, where
$\Delta-F'(\psi_0)$ is invertible. A natural test case beyond this regime is
the cellular flow $\psi_0=\sin x\sin y$ on the flat torus. Here the associated
operator is $\Delta+2$, which has a nontrivial kernel, and therefore the
nondegenerate construction cannot be applied directly.

Our second main construction shows that this degeneracy does not restore the
semilinear elliptic structure. Instead, in a proper subspace of $L^2$, one can still construct smooth Morse steady
states near the cellular flow for which no global relation
$\Delta\psi=F(\psi)$ exists. Moreover, such flexibility phenomena cannot be observed in the odd-odd symmetric class.
\begin{theorem}[Flexibility near the cellular flow]\label{thm:Flexibility near the cellular flow}
There exists a family of smooth Morse steady states $\psi_\epsilon$ converging to $\sin{x}\sin{y}$
in the smooth category as $\epsilon \rightarrow 0$ such that there is no function $F_{\epsilon}$  for which 
\[
\Delta\psi_\epsilon=F_{\epsilon}(\psi_\epsilon).
\]
Moreover, $\psi_{\epsilon}$ is not odd--odd symmetric and is not even
symmetric with respect to any diagonal of $\mathbb T^2$.
 In contrast, within the odd--odd symmetry class, every steady state close to $\psi_0$ solves a semilinear elliptic equation and it is even symmetric to both diagonals of $\mathbb{T}^2$.
\end{theorem}

\subsection{Discussion}
The main question to be asked after Theorem \ref{thm: main theorem} is whether such steady states can be nonlinearly stable. As far as we know, the known nonlinear stability theorems rely on variational principles associated with the conserved quantities \eqref{CLs}, which seems to require the existence of a global functional relation between the stream function $\psi$ and the vorticity $\Delta\psi.$ Since Theorem \ref{thm: main theorem} indicates that we should not always expect the existence of such a global relationship, at least in the smooth category, it is important to expand the nonlinear stability theory to accommodate multi-valued relations between the stream function and the vorticity. For this, we will only be able to propose one direction for future study: namely, the study of nonlinear stability in $L^\infty$ of vorticity. Indeed, in the examples of steady states that we give for which no (single-valued) global relation $F$ exists, the vorticity roughly consists of two bumps with slightly different heights. After a long time, it could be that perturbations cause the two bumps to equalize thus returning to a state with a global $F.$ On the other hand, for this to happen, the solution must first deviate quite a bit from the base steady state in $L^\infty.$ This is because the Euler equation is fundamentally a transport equation and moving particles between the bumps requires particles leaving regions of high vorticity and passing through regions of small vorticity. We note that there are no steady states with non-constant vorticity for which $L^\infty$ stability is known; in fact, to our knowledge, even quantitative stability beyond $L^2$ is unknown in general (even though $L^2$ stability implies $L^p$ stability qualitatively for $2<p<\infty$). Resolving this problem, which is the subject of a question of Yudovich \cite{yudovich2003eleven} (see also \cite{drivas2023singularity}), seems to be the first step towards determining whether general smooth steady states can be nonlinearly stable (even in $L^2$). For the sake of completeness, and to highlight its importance, we state the question clearly:
\begin{question}\label{Q2}
Does there exist any steady state $\omega_*$ on any path connected domain $\Omega$ that is non-constant and also nonlinearly stable in $L^\infty?$
\end{question}
\begin{remark}
In \cite{yudovich2003eleven,drivas2023singularity}, this problem is stated slightly differently. We have elected to state it this way to reflect our (lack of) knowledge on the problem. 
\end{remark}
\begin{remark}
We prove in Proposition \ref{prop:nonlinear stability} that linear stability in the sense of Arnold implies long-time stability in $L^2$ (and can be generalized to $L^\infty$). Whether this can be done globally in time is an interesting question. 
\end{remark}

\subsection{Main conceptual and technical ideas }
\subsubsection{Quasilinearization}\label{sec:Quasilinearization}
One of the contributions of this work is to give a proof of flexibility for 2D Euler steady states relying directly on the linearization of the steady 2D Euler equation rather than the semilinear elliptic equation. Indeed, in the previous works on this subject \cite{choffrut2012local,constantin2021flexibility,coti2023stationary}, the steady Euler equation relied on the semilinear elliptic equation, where it is possible to apply a version of the implicit function theorem in infinite dimensions. Applying the implicit function theorem directly to the steady 2D Euler equations is much more difficult since the equation is quasilinear; this means that there is a loss of derivative that appears when treating the nonlinear term perturbatively. This difficulty was discussed in some detail at the beginning of Section 2 of \cite{coti2023stationary}. To address this issue, the first named author, Drivas, and Ginsberg, proposed a general scheme to "quasi-linearize" the system rather than just linearize it. Let us discuss this in detail. Let $\omega_0$ be the vorticity of a steady state, and assume that $\omega_0+\omega_1$ corresponds to the vorticity of a new steady state (with $\omega_1$ small in some strong topology). Then $\omega_1$ satisfies:
\[\mathcal{L}_0(\omega_1)+u_1\cdot\nabla\omega_1=0,\] with 
$\mathcal{L}_0$ just the linearization around $\omega_0.$ Now, if we were to simply throw the nonlinear term on the right side, we would have:
\begin{equation}\label{badscheme}\mathcal{L}_0(\omega_1)=-u_1\cdot\nabla \omega_1.\end{equation} While \eqref{badscheme} looks good for a fixed point argument based on the smallness of $\omega_1,$ we note that $\mathcal{L}_0^{-1}$ gains a derivative in the direction of $u_0$ while the right side loses a derivative in the direction of $u_1.$ Previous numerical simulations indicate that an iteration based on this does not actually converge (as is expected). However, let us inspect \eqref{badscheme} further by making clear what $\mathcal{L}_0$ is:
\begin{equation}\label{perturbedeqn}u_0\cdot\nabla\omega_1+u_1\cdot\nabla \omega_0+u_1\cdot\nabla\omega_1=0.\end{equation} Now we observe that since $\omega_0$ is a steady state, we define the function 
\[M:=\frac{\nabla\omega_0}{\nabla\psi_0},\] which formally satisfies:
\[\nabla^\perp\psi_0\cdot\nabla M=0,\] using the inverse function theorem (or the straightening lemma). Thus, we observe that \eqref{perturbedeqn} may be written as:
\[u_0\cdot\nabla\omega_1+Mu_1\cdot\nabla\psi_0+u_1\cdot\nabla\omega_1=0,\] which is equivalent to:
\[u_0\cdot\nabla\omega_1-Mu_0\cdot\nabla\psi_1+u_1\cdot\nabla\omega_1=0,\]
so that
\[u_0\cdot\nabla\omega_1-u_0\cdot\nabla (M\psi_1)+u_1\cdot\nabla\omega_1=0,\]
which gives 
\[u_0\cdot\nabla(\omega_1-M\psi_1)+u_1\cdot\nabla\omega_1=0.\] Adding and subtracting the same term, we see that the equation \eqref{badscheme} can be rewritten as:
\begin{equation}\label{goodscheme}(u_0+u_1)\cdot\nabla(\omega_1-M\psi_1)=-u_1\cdot\nabla (M\psi_1).\end{equation}
Now, assuming that the operator on the left is invertible, we see that the loss of derivative problem of \eqref{badscheme} has disappeared and \eqref{goodscheme} can be solved via standard fixed point theorems: either the Banach fixed point theorem with an iteration or directly with the Schauder fixed point theorem. Now the key is simply to show that the invertibility properties of $u_0\cdot\nabla$ are "stable" to the small perturbation by $u_1.$ That the right-hand side is actually in the range of the operator $(u_0+u_1)\cdot\nabla$ follows from the following identity:
\[(u_0+u_1)\cdot\nabla (\omega_0+\omega_1)-u_0\cdot\nabla \omega_0 = (u_0+u_1)\cdot\nabla (\omega_1-M\psi_1)+u_1\cdot\nabla (M\psi_1).\]

\subsubsection{Solving  $(u_0+u_1)\cdot\nabla f =g$}\label{sec:Solving  Hamiltonian}
As we discussed in the previous section, a crucial step in the proof of Theorem \ref{thm: main theorem} is to find a solution to $(u_0+u_1)\cdot\nabla f =g$, where $u_1$ is small. As $u_0$ and $u_1$ are divergence free, it reduces to studying the equation of the form 
\begin{equation}\label{eqn: Hamiltonian flow formula}
    \nabla^{\perp}(H_0+H_1)\cdot\nabla f =g.
\end{equation} 
There are two main challenges. The first comes from verifying the solvability of \eqref{eqn: Hamiltonian flow formula}, and we need to verify that $f$ lies in the range of $ \nabla^{\perp}(H_0+H_1)\cdot\nabla$. The second comes from the kernel of the operator $ \nabla^{\perp}(H_0+H_1)\cdot\nabla$ being infinite dimensional, and we need to choose a right inverse of  $\nabla^{\perp}(H_0+H_1)\cdot\nabla$ and prove it is a bounded operator on H\" older spaces. In this paper, we choose a right inverse of $\nabla^{\perp}(H_0+H_1)\cdot\nabla$  such that it is mean-free on each connected component of the level set of $H_0+H_1$.
\subsubsection{Construction of a Morse steady state with multiple critical points}\label{sec:Construction of a Morse steady states with multiple critical points}
Our first step is to construct a Morse steady state in a simply connected domain with multiple critical points (see the work \cite{lin2004some} for a related but different construction). In particular, we need an example where both parts of Theorem \ref{thm: main theorem} are applicable. The stream function $\psi_0$ of the example we construct solves \[\Delta \psi_0=1+\lambda \psi_0\] for  small $\lambda>0$. Due to the smallness of $\lambda$, $\psi_0$ is close to $\Delta^{-1}(1)$ and therefore the key in this part is to find a domain such that $\Delta^{-1}(1)$ is a Morse function with multiple critical points. We choose the domain to be a non-convex Neumann oval, which is the image of the unit disk under an explicit holomorphic map. In this case, we have an explicit Taylor expression for $\Delta^{-1}(1)$. We explicitly compute the locations of the critical points and analyze the behavior of $\Delta^{-1}(1)$ near them, concluding that $\Delta^{-1}(1)$ is a Morse function with three critical points.     
\subsubsection{The degenerate cellular flow case}\label{sec:Flexibility  and rigidity results near cellular flow}
The cellular flow provides a canonical degenerate setting in which the
nondegenerate theory of Theorem~\ref{thm: main theorem} cannot be applied.
For $\psi_0=\sin x\sin y$, the relevant Schrödinger operator is $\Delta+2$,
which has a nontrivial kernel. This kernel is the main new difficulty in the
construction: one must choose the perturbation and a proper right inverse of the Hamiltonian to ensure the compatibility conditions.

The main result of Section~3 is that this degeneracy does not restore the
semilinear elliptic structure. By performing a Lyapunov-Schmidt argument, we
construct smooth Morse steady states arbitrarily close to the cellular flow
which do not satisfy any global relation \(\Delta\psi=F(\psi)\). The odd--odd
class provides a useful contrast: in that symmetry class, the directions
responsible for breaking the semilinear elliptic structure are excluded, and
the semilinear elliptic structure is preserved.
\subsection{Acknowledgments} The authors thank T. Drivas and D. Ginsberg for helpful comments and previous discussions on this topic. T.M.E. acknowledges partial funding from the NSF DMS-2510472 as well as a Simons Travel Grant. Y. H. acknowledges partial funding from EPSRC Horizon Europe Guarantee EP/X020886/1. 
\subsection{Outline of the paper}
In Section 2, we prove Theorem \ref{thm: main theorem}. We first explain the details in Section \ref{sec:Quasilinearization} and Section \ref{sec:Solving  Hamiltonian} and construct steady states that break the semilinear elliptic equation near the Arnold-stable steady state with approximately flat forcing given in Definition \ref{def:approximate flat}. Then, we construct Arnold-stable steady states in a Neumann oval such that the stream function has multiple critical points and is a Morse function. Next, we show that a general Arnold-stable steady state can be approximated by an Arnold-stable steady state with approximately flat forcing. As a consequence, we finish the first part of Theorem \ref{thm: main theorem}. In the end of this section, we use the Casimir to show that the steady states we constructed are linearly stable. In Section~3, we turn to the degenerate cellular-flow regime. Since the
operator $\Delta+2$ has a nontrivial kernel, the nondegenerate construction
from Section~2 does not apply directly. We perform a Lyapunov-Schmidt argument and construct smooth Morse steady states near the cellular
flow which break the semilinear elliptic structure. We also include a
symmetry-restricted contrast in the odd--odd class. In the appendices, we list facts from complex analysis and action-angle coordinates that we will use in the paper.
\section{On the construction of steady states that break the semilinear elliptic equation structure}
In this section, we first discuss the existence of steady states that break the semilinear elliptic equation structure.
\begin{proposition}\label{thm:main}
    Let $\psi_0$ be a Morse function that solves the following semilinear elliptic equation in a simply connected smooth domain $\Omega$: 
    \begin{equation}
    \begin{aligned}
         &\Delta \psi_0=F(\psi_0)\text{ in $\Omega$}\\&
         \psi_0=0 \text{ in $\partial \Omega$. }
    \end{aligned}
    \end{equation} 
   
  Let $k\geq 6$ be a positive integer. Assume $F$ is $C^{\lfloor\frac{k}{2}\rfloor-1}$, that $\Delta-F^{'}(\psi_0)$ is invertible from $H_{0}^{2}$ to $L^2$ and that $\psi_0$ has more than one critical point. Then there exists a family of smooth steady states $\psi_{\epsilon}$ converging to $\psi_0$ in the $C^{\lfloor\frac{k}{2}-1\rfloor}$ category as $\epsilon \rightarrow 0$ with the property that $\psi_{\epsilon}$ does not solve a semilinear elliptic equation in $\Omega$. Moreover, there is a small $C^{2}$ neighborhood $\mathcal{O}_{\epsilon}$ of $\psi_{\epsilon}$ in which all steady states do not solve a semilinear elliptic equation.
\end{proposition}
\begin{remark}
    By the inverse function theorem, the conditions that $\psi_0$ is a $C^{k}$ Morse function and $\Delta \psi_0=F(\psi_0)$ imply that $F\in C^{\lfloor\frac{k}{2}\rfloor-1}$. As a result, Proposition \ref{thm:main} implies the first part of Theorem \ref{thm: main theorem}.
\end{remark}
\begin{remark}
Examples of domain $\Omega$ and solution $\psi_0$ in Proposition \ref{thm:main} are given in Proposition \ref{thm:construction}. 
\end{remark}
\subsection{Proof of Proposition \ref{thm:main} when $F$ is approximately \textbf{$m$-flat} with respect to $\psi_0$}
\begin{definition}\label{def:approximate flat}
Let $m$ be a positive integer, we say $F$ is approximately \textbf{$m$-flat} if 
    \begin{itemize}
\item $F\in C^\infty$
\item $F^{(m)}$ \footnote{$F^{(m)}$ denotes the $m$-th derivative of $F$.}  vanishes in a neighborhood of critical values of $\psi_0.$ 
\end{itemize}
\end{definition}
In this section, we prove Proposition \ref{thm:main} when $F$ is approximately $(\lfloor\frac{k}{2}\rfloor)-\text{flat}$. In this case, $\psi_0$ is smooth, and we show that $\psi_{\epsilon}$ is  close to $\psi_0$ in the smooth category.
These assumptions will be realized in Section \ref{sec:conclusion}  via a standard approximation argument. 
\begin{proposition}\label{lemma:prop2.1}
    Under the setting of Proposition \ref{thm:main}, if we assume that $F$ is approximately $(\lfloor\frac{k}{2}\rfloor)-\text{flat}$, then there exists a family of smooth functions $\psi_{\epsilon}$ satisfying \begin{equation}
    \begin{aligned}
         &\nabla^{\perp} \psi_{\epsilon}\cdot \nabla \Delta \psi_{\epsilon}=0,\text{in $\Omega$},\\&
        \psi_{\epsilon}=0,\text{in $\partial \Omega$}.
    \end{aligned}
    \end{equation}
     For all positive integers $n$, $\lim _{\epsilon\rightarrow 0}\|\psi_{\epsilon}-\psi_0\|_{C^{n}}=0$. In addition, $\lim_{\epsilon\rightarrow 0}\|\frac{\nabla\Delta\psi_{\epsilon}}{\nabla \psi_{\epsilon}}-F^{'}(\psi_0)\|_{L^{\infty}}=0$.
    Moreover, there exists a neighborhood of $\psi_{\epsilon}$ in $C^2$ where all steady states fail to solve a semilinear elliptic equation.
\end{proposition}
\subsubsection{Construction of the first-order perturbation} 
We now construct the first-order perturbation, which we call $\psi_1.$

  Denote the critical points of $\psi_0$ by $\{x_1,\cdots,x_m\}$ with $m\geq 2$, as $\psi_0$ is a Morse function,  by Morse theory, $\psi_0$ has at least one saddle point. Without loss of generality, we take $x_1$ to be a saddle point. We now construct a perturbation in a neighborhood of $x_1$. Here is an illustrative figure for our construction. \begin{figure}[h!]
 \caption{An illustrative figure for the construction}
    \centering
    \includegraphics[width=0.5\linewidth]{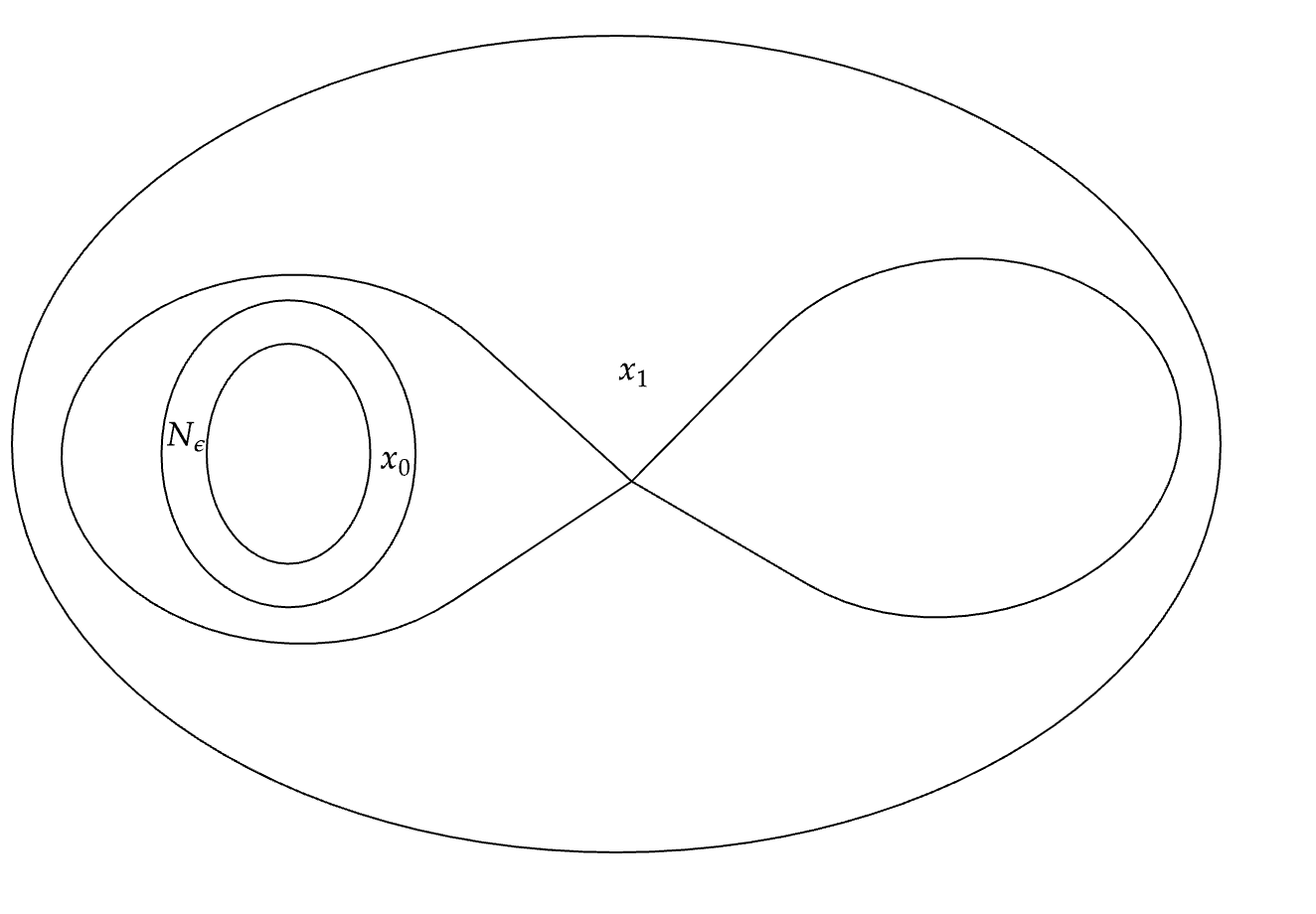}
\end{figure}
  
 In the above figure, $x_0$  is a point close to $x_1$, so that $\psi_0(x_0)$ is a regular value for $\psi_0$ while the set $\{\psi_0=\psi_0(x_0)\}$ has multiple components. Now we fix a $\epsilon_0>0$ such that there is a connected component of $\{\psi_0(x_0)-\epsilon_0\leq \psi_0\leq  \psi_0(x_0)+\epsilon_0\}$  containing $x_0$ but no critical points $x_1,\cdots,x_m$ (in particular, the set only contains one connected component of $\{\psi_0=\psi_0(x_0)\}$). We denote the set by $\mathcal{N}_{\epsilon_0}$. Now we fix a smooth function $G\not\equiv 0$ with $G(t) =0$ for all $|t-\psi_0(x_0)|\geq \epsilon_0$.

Correspondingly, we can define a smooth function $\eta$:\begin{align*}
    \eta(x):=&\left\{\begin{array}{cc}
    0,\quad & x\in \mathcal{N}_{\epsilon_0}^{c}\cap \Omega;  \\
        G(\psi_0), &\quad  x\in \mathcal{N}_{\epsilon_0}.\\
    \end{array}\right.
\end{align*} 
$\eta$ is a smooth function in $\Omega$. We now define:
\begin{equation}\label{eqn:defpsi1}
\psi_1=(\Delta-F'(\psi_0))^{-1}\eta    
\end{equation}
 and we have the following proposition concerning $\psi_1$. 
 \begin{proposition}\label{Prop:psi1regularity}
     $\psi_1$ is a smooth function satisfying \begin{equation}\label{eqn:appoximate psi1}
\begin{aligned}
     \nabla^{\perp} \psi_{0}\cdot \nabla(\Delta\psi_{1}-F^{'}(\psi_0)\psi_{1})+\nabla^{\perp} \psi_{1}\cdot\nabla(\Delta\psi_{0}-F^{'}(\psi_0)\psi_{0})=-F^{''}(\psi_0)\psi_0\nabla^{\perp}\psi_1\cdot \nabla \psi_0.
\end{aligned}
\end{equation}
Moreover, for all positive integers $m$ \begin{equation}
    |\psi_1|_{C^{m}}\lesssim |\psi_0|_{C^{m+2}}.\footnote{In this paper, we say $A\lesssim B$ if there is a universal constant C such that $|A|\leq C|B|$, and $A\lesssim_{\delta} B$, if there is a constant C depending on $\delta$ such that $|A|\leq C|B|$ .}
\end{equation}
 \end{proposition}
The equation \eqref{eqn:appoximate psi1} is equivalent to the $O(\epsilon)$ term in the expansion of $\nabla^{\perp}\psi_{\epsilon}\cdot \nabla \Delta \psi_{\epsilon}=0$.  The regularity estimates concerning $\psi_1$ follow from standard arguments in elliptic theory.

\subsubsection{Construction of the remainder}
Given the first-order perturbation $\psi_1$ constructed in the above section, we now state a proposition giving the remainder $\psi_2.$
\begin{proposition}\label{prop:C1 approximation}
     There exists a constant $C>0$ depending on $\psi_0,F,$ and $\eta$ such that for all $\epsilon>0$ sufficiently small, there is a $C^1$ function $\psi_2$ solving  
     \begin{equation}\label{eqn:steady full eqn}
         \begin{aligned}
            \nabla^{\perp} (\psi_0+\epsilon\psi_1+\epsilon^2\psi_2) \cdot \nabla \Delta (\psi_0+\epsilon\psi_1+\epsilon^2\psi_2)=0
         \end{aligned}
     \end{equation} 
with $\|\psi_2\|_{C^1(\Omega)}\leq C$.
\end{proposition}
Let $\psi_{\epsilon}=\psi_0+\epsilon \psi_1+\epsilon^2 \psi_2$, then
\eqref{eqn:steady full eqn} is equivalent to 
\begin{equation}\label{eqn:second order approxmimation}
    \begin{aligned}
       &\nabla^{\perp}\psi_{\epsilon}\cdot \nabla (\Delta \psi_2-F^{'}(\psi_0)\psi_2)\\&
        =-F^{''}(\psi_0)[\epsilon \psi_2 \nabla^{\perp}\psi_1 \cdot \nabla \psi_0+\epsilon^2 \psi_2 \nabla^{\perp}\psi_2 \cdot \nabla \psi_0+\psi_1\nabla^{\perp}\psi_1 \cdot \nabla \psi_0+\epsilon \psi_1 \nabla^{\perp}\psi_2 \cdot \nabla \psi_0 ]\\&\quad -(\nabla^{\perp}\psi_{1}+\epsilon \nabla^{\perp}\psi_2)\cdot \nabla [\Delta \psi_{1}-F^{'}(\psi_0)\psi_1]
        :=g.
    \end{aligned}
\end{equation}
 If we define:
\[S=\nabla^\perp\psi_\epsilon\cdot\nabla,\qquad A=\Delta-F'(\psi_0),\] then \eqref{eqn:second order approxmimation} can be written schematically as:
\[SA\psi_2=g.\] 
We have the following qualitative statement concerning $g$.
\begin{proposition}\label{prop:g qualitative}
    $g$ lies in the range of the operator $S$. In addition, $g$ admits the following decomposition.
\begin{equation}\label{eqn:g-qualitative-decomposition1}
    g =F''(\psi_{0})\,
\nabla^{\perp}\psi_{\epsilon}\cdot
\nabla\left(\dfrac{(\psi_{1}+\epsilon\psi_{2})^{2}}{2}\right)
-\mathcal{R},
\end{equation}
where \[\mathcal{R}
:=
(\nabla^{\perp}\psi_{1}
+\epsilon\nabla^{\perp}\psi_{2})
\cdot
\nabla\left[
\Delta\psi_{1}
-F'(\psi_{0})\psi_{1}
\right].
\]
Moreover, for $k\geq 6$, denote $m=\lfloor\frac{k}{2}\rfloor$, we have 
    \begin{equation}\label{eqn:g-qualitative-decomposition2}
\begin{aligned}
g&=\nabla^{\perp}\psi_{\epsilon}\cdot
\nabla\left(
\displaystyle\sum_{j=2}^{m-1}
\epsilon^{j-2}F^{(j)}(\psi_{0})
\dfrac{(\psi_{1}+\epsilon\psi_{2})^{j}}{j!}
\right) \\&
\qquad
+\epsilon^{m-2}F^{(m)}(\psi_{0})\,
\nabla^{\perp}\psi_{\epsilon}\cdot
\nabla\left(
\dfrac{(\psi_{1}+\epsilon\psi_{2})^{m}}{m!}
\right)
-\mathcal{R},
\end{aligned}
\end{equation}
\end{proposition}
\begin{proof}
    Due to \eqref{eqn:appoximate psi1}, we have \begin{equation}
    \begin{aligned}
        g&=\frac{1}{\epsilon^2}[\nabla^{\perp}\psi_{\epsilon}\cdot\nabla (\Delta\psi_{\epsilon}-F^{'}(\psi_0)\psi_{\epsilon})-\nabla^{\perp} \psi_{0}\cdot\nabla(\Delta\psi_{0}-F^{'}(\psi_0)\psi_{0})\\&\quad -\epsilon\nabla^{\perp} \psi_{0}\cdot \nabla(\Delta\psi_{1}-F^{'}(\psi_0)\psi_{1})-\epsilon\nabla^{\perp} \psi_{1}\cdot\nabla(\Delta\psi_{0}-F^{'}(\psi_0)\psi_{0})+\epsilon F^{''}(\psi_0)\psi_0\nabla^{\perp}\psi_1\cdot \nabla \psi_0\\&\quad -\epsilon^2 \nabla^{\perp} \psi_{\epsilon}\cdot (\Delta\psi_2-F^{'}(\psi_0)\psi_2) ]\\&=\frac{1}{\epsilon^2}\nabla^{\perp}\psi_{\epsilon}\cdot\nabla (\Delta\psi_{0}-F^{'}(\psi_0)\psi_{0}+\epsilon \Delta\psi_{1}-\epsilon F^{'}(\psi_0)\psi_{1}),
    \end{aligned}
\end{equation}
$g$ is in the range of $S$ (which is a qualitative statement).
Moreover, note that \begin{equation}
    g=-F^{''}(\psi_0)\nabla^{\perp}\psi_{\epsilon}\cdot \nabla(\psi_0)(\psi_2+\frac{\psi_1}{\epsilon})-(\nabla^{\perp}\psi_{1}+\epsilon \nabla^{\perp}\psi_2)\cdot \nabla [\Delta \psi_{1}-F^{'}(\psi_0)\psi_1],
\end{equation}  using the identity $\nabla^{\perp} \psi_{\epsilon}\cdot \nabla \psi_{\epsilon}=0$, we obtain \begin{equation}
\begin{aligned}
      g=&F^{''}(\psi_0)\nabla^{\perp}\psi_{\epsilon}\cdot \nabla(\epsilon \psi_1+\epsilon^2\psi_2)(\psi_2+\frac{\psi_1}{\epsilon})-(\nabla^{\perp}\psi_{1}+\epsilon \nabla^{\perp}\psi_2)\cdot \nabla [\Delta \psi_{1}-F^{'}(\psi_0)\psi_1]\\&=F''(\psi_{0})\,
\nabla^{\perp}\psi_{\epsilon}\cdot
\nabla\left(\dfrac{(\psi_{1}+\epsilon\psi_{2})^{2}}{2}\right)
-\mathcal{R},
\end{aligned}
\end{equation}
which is \eqref{eqn:g-qualitative-decomposition1}.
Similarly, iteratively using the identity $\nabla^{\perp} \psi_{\epsilon}\cdot \nabla \psi_{\epsilon}=0$, we have \eqref{eqn:g-qualitative-decomposition2}. 
\end{proof} Now it suffices to prove good quantitative bounds for a right inverse of $S.$ We will discuss this issue and prove Proposition \ref{prop:C1 approximation} in the next section.

\subsubsection{A right inverse of the Hamiltonian transport operator and proof of Proposition \ref{prop:C1 approximation}}

Let $\Omega$ be a compact domain in $\mathbb{R}^2$ and let $H\in C^{1}(\Omega)$ such that $H$ is constant on connected components of $\partial \Omega$. Define the Hamiltonian vector field
\[
X_H:=\nabla^\perp H:=(-\partial_2 H,\partial_1 H),
\]
and the first-order transport operator
\[
\mathcal L_H f := X_H\cdot \nabla f .
\]
We are interested in solving
\[
\mathcal L_H f = g,
\]
where $g$ lies in the range of $\mathcal L_H$. In addition, we assume
$\supp(g)$ has positive distance from $H^{-1}\{\nabla H(x)=0\}$, which is the critical level set of $H$.

\medskip
Fix $x_0\in\Omega$ such that
$H(x_0)$ is a regular value for $H$. 
Let $\Sigma_{x_0}$ be the connected component of the level set of $H$ containing $x_0$ and $\Gamma(x_0,t)$ be a flow generated by $X_H$:
\begin{equation}\label{eq:flow}
\begin{aligned}
\frac{\partial }{\partial t}\Gamma(x_0,t) &= X_H(\Gamma(x_0,t)),\\
\Gamma(x_{0},0) &= x_0.
\end{aligned}
\end{equation}
$\Gamma(x_0,\cdot)$ gives a periodic parametrization of $\Sigma_{x_0}$ and  we denote its (minimal) period by
\[
T(x_0)>0.
\]
We normalize the parametrization by defining, for $t\in[0,1]$,
\begin{equation}\label{eqn:rescaled H flow}
    \widetilde{\Gamma}(x_{0},t):=\Gamma(x_0,T(x_0)\,t).
\end{equation}
For each $y\in \Sigma_{x_0}$, there exists a unique $t_y\in[0,1)$ such that
\[
\widetilde{\Gamma}(x_0,t_y)=y.
\]
Note that $g$ lying in the range of $\mathcal{L}_H$ implies that
\[
\int_{0}^{1} g\big(\widetilde{\Gamma}(x_0,t\big)\,dt = 0.
\]
We now define the primitive of $g$ in $\Sigma_{x_0}$ with zero mean.
\begin{definition}\label{def:Mx0}
For $y\in\Sigma_{x_0}$, define
\begin{equation}\label{eqn:defM}
\mathcal M_{x_0}g(y)
:=\int_{0}^{t_{y}} g\big(\widetilde{\Gamma}(x_{0},s)\big)\,ds
-\int_{0}^{1}\left(\int_{0}^{t} g\big(\widetilde{\Gamma}(x_{0},s)\big)\,ds\right)dt .
\end{equation}
\end{definition}

It is easy to see that the quantity $\mathcal M_{x_0}g(y)$ is independent of the choice of base point $x_0$. We now define $\mathcal L_H^{-1}$  below.

\begin{definition}\label{def:LHinv}
Assume $g\in \mathrm{Ran}(\mathcal L_H)$ and $\supp(g)$ has positive distance from $H^{-1}\{\nabla H(x)=0\}$. We define
\begin{equation}\label{eqn:defH-1}
\mathcal L_H^{-1}g(x):=
\begin{cases}
T(x)\mathcal M_{x}g(x), & \text{if $x\notin H^{-1}\{\nabla H(x)=0\}$,}\\[3pt]
0, & \text{if $x\in H^{-1}\{\nabla H(x)=0\}$.}
\end{cases}
\end{equation}
\end{definition}

We now verify $\mathcal L_H^{-1}$ is a right inverse of $\mathcal L_H$.
\begin{proposition}\label{lemma:solve transport}
Let $\mathcal{L}_{H}:=\nabla^\perp H\cdot \nabla$, and $H$ be a $C^{1}$ function. Assume that
$g$ belongs to the range of $\mathcal{L}_{H}$ and there exists a positive constant $a$ such that 
\[
|\nabla H|>a,\text{for all $z\in H^{-1}(H(\operatorname{supp}g))$},
\]
then the following properties hold for $\mathcal{L}_{H}^{-1}$.

\begin{itemize}
    \item[(1)] There exists a solution $f=\mathcal{L}_{H}^{-1}g$ to
    \[
    \mathcal{L}_{H}f=g
    \]
    such that
    \begin{equation}
        \|\mathcal{L}_{H}^{-1}g\|_{L^\infty}\leq 2|T|_{L^{\infty}(\{H^{-1}\bigl(H(\operatorname{supp}g\bigr))\})} \|g\|_{L^\infty}.
    \end{equation}

    \item[(2)] Assume in addition that $H\in C^{1,1}(\{\nabla H\neq 0\})$ and $g\in C^{0,1}$,
    then
    \[
    \mathcal{L}_{H}^{-1}g\in C^{0,1}.
    \]

    \item[(3)] Let \(k\in \mathbb{N}\) and \(\alpha\in(0,1)\). Assume that
\(H\in C^{k,\alpha}\) and \(g\in C^{k-1,\alpha}\). Then there exists
\(\epsilon_{1}>0\) such that, for every
\(\tilde H\in C^{k,\alpha}\) satisfying
\[
\|\tilde H-H\|_{C^{k,\alpha}}\leq \epsilon_{1},
\]
if
\[
|\nabla \tilde H|>\frac a2
\quad\text{on}\quad
\tilde H^{-1}\bigl(\tilde H(\operatorname{supp}g)\bigr)
\]
and \(g\) belongs to the range of \(L_{\tilde H}\), then
\[
\|L_{\tilde H}^{-1}g\|_{C^{k-1,\alpha}}
\leq C\|g\|_{C^{k-1,\alpha}},
\]
where \(C\) is a constant depending  on \(H,k,a,\alpha\).
\end{itemize}
\end{proposition}
   \begin{proof}
     Property (1) follows directly from \eqref{eqn:defM} and \eqref{eqn:defH-1}. 
     For property (2), let $x_0$ be a fixed point in $\Omega$ and $\Sigma_{x_0}$ be the connected component of $\{H=H(x_0)\}$ crossing $x_0$,  it suffices to verify $\mathcal{L}_{H}^{-1} g$ is Lipschitz near $\Sigma_{x_0}$. First, due to Proposition \ref{prop:regularity period}, the periodic map $T$ is Lipschitz near $\Sigma_{x_0}$. Now, choose $n$ to be the direction of $\nabla H$ at $x_0$, we can parametrize a tubular neighborhood of $\Sigma_{x_0}$ by the rescaled Hamiltonian flow:
    \begin{equation}
    \begin{aligned}
         &\frac{dX(s,t)}{dt}=T(X(s,t))\nabla^{\perp} H(X(s,t))\\&
         X(s,0)=x_0+sn.
    \end{aligned}
    \end{equation}
     $T\nabla^{\perp} H(\cdot)$ is Lipschitz, with $T\nabla^{\perp} H(\cdot)\neq 0$ at $x_0$ , then there exists a positive $\epsilon$ such that $X$ is a bi-Lipschitz mapping from $(-\epsilon,\epsilon)\times [0,1]$ to its image, which is a tubular neighborhood near $\Sigma_{x_0}$. We denote $X^{-1}(y)=(r(y),\tau(y))$, then $t_{y}$ near $\Sigma_{x_0}$ is given by $t_{y}:=\tau(y)$, which is a Lipschitz map in $y$. 
      By \eqref{eqn:defM},
     \begin{equation}
\mathcal M_{y}g(y)
:=\int_{0}^{\tau(y)} g\big({\Gamma}(r(y),(T(y)s)\big)\,ds
-\int_{0}^{1}\left(\int_{0}^{t} g\big({\Gamma}(r(y),T(y)s\big)\,ds\right)dt.
\end{equation}
As $\Gamma$ is a Lipschitz map, we have $\mathcal{L}_{H}^{-1}$ is Lipschitz near $\Sigma_{x_0}$.
Analogously, Property (3) follows from the inverse function theorem, \eqref{eqn:defM} and Proposition \ref{prop:regularity period}.
   \end{proof}
   
We now use Proposition~\ref{lemma:solve transport} to prove the existence of steady states in Proposition \ref{prop:C1 approximation}.
\begin{proof}[Proof of Proposition \ref{prop:C1 approximation}]
Let $C_0$ be a positive constant that we will choose later in the proof which depends on $\psi_0$ and $\psi_1$. As 
$\supp F^{(\lfloor \frac{k}{2}\rfloor)}(\psi_0)\ \cup\ \mathcal N_{\epsilon_0}$ is a compact set away from the critical level set of  $\psi_0$,  there exists a positive constant $\epsilon_0$ depending on $C_0$, $\psi_1$ and $\psi_0$ such that for all $0<\epsilon\leq \epsilon_0$ and $\|\psi_2\|_{C^1}\leq C_0$, the operator 
\begin{equation}
\begin{aligned}
    \mathcal{H}_{\epsilon}:=\mathcal{L}_{\psi_{\epsilon}}^{-1}[\epsilon^{\lfloor \frac{k}{2}-2\rfloor} F^{(\lfloor\frac{k}{2}\rfloor)}(\psi_0)\nabla^{\perp} \psi_{\epsilon}\cdot \nabla(\frac{(\psi_1+\epsilon\psi_2)^{\lfloor\frac{k}{2}\rfloor}}{(\lfloor\frac{k}{2}\rfloor )!})-(\nabla^{\perp}\psi_{1}+\epsilon \nabla^{\perp}\psi_2)\cdot \nabla \eta ]
\end{aligned}   
\end{equation}
is well-defined from $C^1$ to $C^1$, where $\mathcal L^{-1}_{\psi_{\epsilon}}$ is the inverse defined in Definition \ref{def:LHinv}.

Based on \eqref{eqn:second order approxmimation} and Proposition \ref{prop:g qualitative}, the fixed point of the map $\mathcal{L}_0$ given below corresponds to a solution to \eqref{eqn:steady full eqn}:  
\begin{equation}\label{eqn:Llong}
\begin{aligned}
\mathcal L_0\psi_2
:=
(\Delta-F'(\psi_0))^{-1}
\left[
\sum_{j=2}^{m-1}
\epsilon^{j-2} F^{(j)}(\psi_0)
\frac{(\psi_1+\epsilon\psi_2)^j}{j!}
+\mathcal H_{\epsilon}
\right],
\qquad
m:=\left\lfloor \frac{k}{2}\right\rfloor.
\end{aligned}
\end{equation}
Note that $\|\epsilon^{\lfloor \frac{k}{2}-2\rfloor} F^{(\lfloor\frac{k}{2}\rfloor)}(\psi_0)\nabla^{\perp} \psi_{\epsilon}\cdot \nabla(\frac{(\psi_1+\epsilon\psi_2)^{\lfloor\frac{k}{2}\rfloor}}{(\lfloor\frac{k}{2}\rfloor )!})-(\nabla^{\perp}\psi_{1}+\epsilon \nabla^{\perp}\psi_2)\cdot \nabla \eta \|_{L^{\infty}}\lesssim_{\psi_0,\psi_1,k}(1+\epsilon\|\psi_2\|_{C^1}+(\epsilon\|\psi_2\|_{C^1})^{\lfloor\frac{k}{2}\rfloor})$.
Now by the explicit form of $\mathcal{H}_{\epsilon}$, part (1) in Proposition~\ref{lemma:solve transport} and elliptic regularity theory,   \begin{equation}\label{eqn:psi2identity1decompose}
    \|(\Delta-F'(\psi_0))^{-1}
\mathcal H_{\epsilon}\|_{C^{1,\frac{1}{2}}}\lesssim_{\psi_0,\psi_1,k} [1+\epsilon\|\psi_2\|_{C^1}+(\epsilon\|\psi_2\|_{C^1})^{\lfloor\frac{k}{2}\rfloor}].
\end{equation} 
In addition, by elliptic regularity, we have for all $k\geq 6,$ \begin{equation}\label{eqn:psi2identity2decompose}
    \|(\Delta-F^{'}(\psi_0))^{-1}[\sum_{j=2}^{m-1}
\epsilon^{j-2} F^{(j)}(\psi_0)
\frac{(\psi_1+\epsilon\psi_2)^j}{j!}]\|_{C^{1,\frac{1}{2}}}\lesssim_{\psi_0,\psi_1,k} [1+\epsilon\|\psi_2\|_{C^1}+(\epsilon\|\psi_2\|_{C^1})^{\lfloor\frac{k}{2}\rfloor}].
\end{equation}
Now by \eqref{eqn:Llong},\eqref{eqn:psi2identity1decompose}  and \eqref{eqn:psi2identity2decompose}, there exists a constant $C$ depending only on $\psi_0,\psi_1,k$, such that \begin{equation}
    \|\mathcal{L}_0\psi_2\|_{C^{1,\frac{1}{2}}}\leq C[1+\epsilon\|\psi_2\|_{C^1}+(\epsilon\|\psi_2\|_{C^1})^{\lfloor\frac{k}{2}\rfloor}].
\end{equation}  
In particular, choose $C_0=2C$, for sufficiently small $\epsilon$,
the map $\mathcal L_0$ is a compact map on the closed convex set 
$\mathcal K_{C_0}:=\Bigl\{\psi_2:\ \|\psi_2\|_{C^1}\le C_0\Bigr\}$ to itself.
Moreover, similarly to the argument above, by the explicit form of $\mathcal{L}_0$, we have $\mathcal{L}_0\psi_2$ continuously depends on $\psi_2$ in $C^{1}$ topology.
Therefore, by the Schauder fixed point theorem, there exists $\psi_2\in\mathcal K_{C_0}$ such that
\[
\mathcal L_0\psi_2=\psi_2.
\]
\end{proof}
In the next section, we discuss the convergence of $\psi_{\epsilon}$ to $\psi_0$ in the smooth category.
\subsubsection{Conclusion of the proof of Proposition \ref{lemma:prop2.1}}
We first prove that the steady state constructed in Proposition~\ref{prop:C1 approximation} is smooth.

\begin{lemma}\label{prop:smoothness-psi2}
The function $\psi_2$ constructed in Proposition~\ref{prop:C1 approximation} is smooth. Moreover, for every $k\in\mathbb N$ there exists $\epsilon_k>0$ such that
\[
\sup_{0<\epsilon\le \epsilon_k}\|\psi_2\|_{C^k(\Omega)}<\infty .
\]
\end{lemma}

\begin{proof}

 From \eqref{eqn:Llong}, we have $\|(\Delta-F^{'}(\psi_0))\psi_2\|_{L^{\infty}}\lesssim 1 $. Now, as $\Delta-F^{'}(\psi_0)$ is invertible from $H_0^2$ to $L^2$, and $F^{'}(\psi_0)$ is smooth, by elliptic regularity theory, we have 
\[
\|\psi_2\|_{L^\infty(\Omega)} \lesssim 1,
\]
uniformly as $\epsilon\to 0$. In particular,
\[
\bigl\|\Delta(\psi_0+\epsilon\psi_1+\epsilon^2\psi_2)\bigr\|_{L^\infty(\Omega)}\lesssim 1.
\]
Therefore, by Corollary 2.2 in \cite{Nadirashvili2013StationaryEuler}, the function
\[
\psi_\epsilon:=\psi_0+\epsilon\psi_1+\epsilon^2\psi_2
\]
is of class $C^{1,1}$ near the regular points of $\psi_\epsilon$.

Based on the proof of Proposition \ref{prop:C1 approximation}, we have \begin{equation}\label{eqn:psi2identityfix}
    \psi_2=
(\Delta-F'(\psi_0))^{-1}[\sum_{j=2}^{\lfloor\frac{k}{2}-1\rfloor}\epsilon^{j-2} F^{(j)}(\psi_0) \frac{(\psi_1+\epsilon\psi_2)^{j}}{j!}+
\,\mathcal{H}_{\epsilon}].
\end{equation}
Since $\psi_\epsilon$ is $C^{1,1}$ near its regular points and the distance  between $\supp(\mathcal{L}_{\psi_{\epsilon}}\mathcal{H}_{\epsilon})$ and the critical level set of $\psi_{\epsilon}$ (by the hypothesis of Proposition~\ref{prop:C1 approximation}) has non-trivial positive lower bound as $\epsilon\rightarrow 0$, Proposition~\ref{lemma:solve transport} yields that
$\mathcal{H}_{\epsilon}$
is Lipschitz (with a uniform bound for  small $\epsilon$). From \eqref{eqn:psi2identityfix}, $(\Delta-F'(\psi_0))\psi_2$ is uniformly bounded in $C^{0,1}$, then  elliptic regularity gives
\[
\psi_2 \in C^{2,\alpha}(\Omega)\qquad \text{for every }\alpha\in(0,1),
\]
with uniform bounds for sufficiently small $\epsilon$. Consequently,
\[
\psi_\epsilon=\psi_0+\epsilon\psi_1+\epsilon^2\psi_2 \in C^{2,\alpha}(\Omega)
\qquad\text{for every }\alpha\in(0,1).
\]

With this improved regularity, we apply Proposition~\ref{lemma:solve transport} again to conclude that for all $\alpha$,
$(\Delta-F'(\psi_0))\psi_2$ has a uniform upper bound for $C^{1,\alpha}$ in $\epsilon$.
Then, elliptic Schauder estimates imply
\[
\psi_2\in C^{3,\alpha}(\Omega)\qquad\text{for every }\alpha\in(0,1),
\]
 with uniform bounds for $\epsilon$ sufficiently small.

Iterating the above arguments, we obtain that $\psi_2$ is smooth. Moreover, for each $k\in\mathbb N$ there exists $\epsilon_k>0$ such that
\[
\sup_{0<\epsilon\le \epsilon_k}\|\psi_2\|_{C^k(\Omega)}<\infty .
\]
\end{proof}
Now that we have constructed smooth steady states $\psi_{\epsilon}$ near $\psi_0$, to conclude the proof of Proposition \ref{lemma:prop2.1}, it remains to show $\psi_{\epsilon}$ fails to solve a semilinear elliptic equation and establish the convergence of $\frac{\nabla \Delta \psi_{\epsilon}}{\nabla \psi_{\epsilon}}$.
\begin{lemma}\label{lemma:Conclusion of the proof of Proposition}
   $\psi_{\epsilon}$ fails to solve a semilinear elliptic equation and  $\lim_{\epsilon\rightarrow 0}\|\frac{\nabla\Delta\psi_{\epsilon}}{\nabla \psi_{\epsilon}}-F^{'}(\psi_0)\|_{L^{\infty}}=0$.
\end{lemma}
\begin{proof}

 The smooth convergence of $\psi_{\epsilon}$ follows from Lemma \ref{prop:smoothness-psi2} and \eqref{eqn:Llong}.
Moreover, by \eqref{eqn:defpsi1}, we have \begin{equation}\label{eqn:psifullzeroorder1}
    \begin{aligned}
       \Delta \psi_{\epsilon}-F^{'}(\psi_0) \psi_{\epsilon}=[F(\psi_0)-F^{'}(\psi_0)\psi_0]+\epsilon \eta+\epsilon^2\mathcal{H}_{\epsilon}+\sum_{j=2}^{\lfloor\frac{k}{2}-1\rfloor} \epsilon^{j}F^{(j)}(\psi_0)\frac{(\psi_1+\epsilon\psi_2)^{j}}{j!}.
    \end{aligned}
\end{equation}
Now let $\Gamma_1:=\{\psi_{\epsilon}=\psi_{0}(x_0)\}\cap \mathcal{N}_{\epsilon_0}$, $\Gamma_2$ be a connected component of $\{\psi_{\epsilon}=\psi_{0}(x_0)\}\cap \mathcal{N}_{\epsilon_0}^{c}$.
By applying the inverse function theorem to \(\psi_0\) near the regular level set
\[
\{\psi_0=\psi_0(x_0)\},
\]
and using the uniform \(C^2\)-bound on \(\psi_2\), we see that \(\Gamma_1\) and \(\Gamma_2\) lie within \(O(\epsilon)\) of two connected components of the level set \(\{\psi_0=\psi_0(x_0)\}\).
 In particular, we have $\psi_{0}=\psi_{0}(x_0)+O(\epsilon)$ on $\Gamma_1$ and $\Gamma_2$.  Now, by \eqref{eqn:psifullzeroorder1} and $F\in C^2$, we have  $\Delta \psi_{\epsilon}|_{\Gamma_1}-\Delta \psi_{\epsilon}|_{\Gamma_2}=\epsilon \eta(x_0)+O(\epsilon^2)\neq 0$. $\psi_{\epsilon}$ fails to solve a semilinear elliptic equation. Similarly,   all the steady states whose stream function lies in the set $\|f-\psi_{\epsilon}\|_{C^2}\leq \epsilon^3$ fail to solve a semilinear elliptic equation when $\epsilon$ is small.
As a direct consequence of \eqref{eqn:psifullzeroorder1}, we have 
\begin{equation}\label{eqn:psifullone1}
    \begin{aligned}
        \nabla[\Delta \psi_{\epsilon}]&=F^{'}(\psi_0)\nabla \psi_{\epsilon}-F^{''}(\psi_0)\nabla \psi_0(\epsilon\psi_1+\epsilon^2\psi_2)+\epsilon\nabla\eta+\epsilon^2 \nabla(\mathcal{H_{\epsilon}})\\&\quad +\sum_{j=2}^{\lfloor\frac{k}{2}-1\rfloor} \epsilon^{j}F^{(j+1)}(\psi_0)\frac{(\psi_1+\epsilon\psi_2)^{j}}{j!}\nabla\psi_{0}+\epsilon^{j-1}F^{(j)}(\psi_0)\frac{(\psi_1+\epsilon\psi_2)^{(j-1)}}{(j-1)!}\nabla(\psi_{\epsilon}-\psi_{0})\\ & =\sum_{j=1}^{\lfloor\frac{k}{2}-1\rfloor}\epsilon^{j-1}F^{(j)}(\psi_0)\frac{(\psi_1+\epsilon\psi_2)^{(j-1)}}{(j-1)!}\nabla\psi_{\epsilon}+\epsilon\nabla\eta+\epsilon^2 \nabla(\mathcal{H_{\epsilon}})\\&\quad +\epsilon^{\lfloor\frac{k}{2}\rfloor-1}F^{(\lfloor\frac{k}{2}\rfloor)}(\psi_0)\frac{(\psi_1+\epsilon\psi_2)^{(\lfloor\frac{k}{2}\rfloor-1)}}{(\lfloor\frac{k}{2}\rfloor-1)!}\nabla\psi_{0}.
    \end{aligned}
\end{equation}
As $F$ is \textbf{$\lfloor\frac{k}{2}\rfloor\text{-flat}$} with respect to $\psi_0$ and the support of $\eta$ is away from the critical level set of $\psi_0$, we can choose $\xi$  such that for sufficiently small $\epsilon$, the set $\mathcal{M}:=\psi_0^{-1}(\psi_0(\{|\nabla \psi_0|<\xi\}))$  contains no element in $\mathcal{N}_{\epsilon_0}$, and element in the support of $F^{\lfloor\frac{k}{2}\rfloor}(\psi_0)$ and $\mathcal{H}_{\epsilon}$. $\mathcal{M}$ is an open set around the critical points of $\psi_0$.
In particular, due to \eqref{eqn:psifullone1}, we have for all $x\in \mathcal{M}$, \begin{equation}
    \frac{\nabla\Delta\psi_{\epsilon}}{\nabla \psi_{\epsilon}}(x)=\sum_{j=1}^{\lfloor\frac{k}{2}-1\rfloor}\epsilon^{j-1}F^{(j)}(\psi_0)\frac{(\psi_1+\epsilon\psi_2)^{(j-1)}}{(j-1)!}(x).
\end{equation}
Moreover, based on the uniform bound of $\psi_2$ in $C^{2}$, we have in $\mathcal{M}^{c}$, \begin{equation}
    |\frac{\nabla\Delta\psi_{\epsilon}}{\nabla \psi_{\epsilon}}(x)-F^{'}(\psi_0)|\lesssim_{\xi,\psi_0,\psi_1,\eta} \epsilon.
\end{equation}
Combined the two cases discussed above, we have the $L^{\infty}$ convergence of $\frac{\nabla \Delta \psi_{\epsilon}}{\nabla \psi_\epsilon}$ to $F^{'}(\psi_0)$ as $\epsilon \rightarrow 0$.
    \end{proof}
In the following section, we explicitly construct a smooth Arnold-stable steady state and verify that the stream function is a Morse function with multiple critical points.
\subsection{Construction of an Arnold-stable steady state which is a Morse function and has multiple critical points}

For \(q\in(0,1)\), define
\begin{equation}
    f_q(z)=\frac{z}{1-qz^2}=(f_1(x,y),f_2(x,y)).
\end{equation}
It is known that \(f_q\) maps \(\mathbb D\) conformally onto a simply
connected domain \(\Omega_q\), called the Neumann oval. In this section, we
construct an Arnold-stable steady state on \(\Omega_q\) whose stream function
is a Morse function with multiple nondegenerate critical points.

\begin{proposition}\label{thm:construction}
For almost every \(q\in(\sqrt2-1,1)\), there exists \(\lambda>0\) such that
the unique solution \(\psi^\lambda\) to
\begin{equation}\label{eqn:u def}
\begin{aligned}
    \Delta \psi^\lambda &= 1+\lambda\psi^\lambda
        &&\text{in }\Omega_q,\\
    \psi^\lambda &=0
        &&\text{on }\partial\Omega_q
\end{aligned}
\end{equation}
is a Morse function with multiple critical points.
\end{proposition}

In the case where $\lambda=0$, the functions with constant Laplacian on Neumann ovals were previously studied in
\cite{lundberg2021note}. In particular, upon 
reparametrizations, their result shows that, for $q\in (\sqrt{2}-1,1)$, the
solution to
\[
    \Delta\psi_q=1\quad\text{in }\Omega_q,\qquad
    \psi_q=0\quad\text{on }\partial\Omega_q
\]
has three critical points. We include a self-contained derivation of this fact
in our notation and further show that for almost every \(q\) in the above
range, these critical points are nondegenerate, so that \(\psi_q\) is a Morse
function. We then perturb this constant-Laplacian solution to construct
solutions of \eqref{eqn:u def} close to \(\psi_q\), preserving the Morse
structure and the existence of multiple critical points.

\subsubsection{Property of the Neumann ovals and topological behavior of functions with constant Laplacian in  Neumann ovals}
We first discuss various properties of the Neumann ovals.
Based on Proposition \ref{prop: convex}, we have 
\begin{lemma}
For $q\in (0,1)$, $f_{q}(z)$ is univalent from $\mathbb{D}$ to its image.   $f_{q}$ maps the \(x\)-axis  to the \(x\)-axis  and  the \(y\)-axis to the \(y\)-axis. $f_{q}(\mathbb{D})$ is symmetric with respect to both the \(x\)- and \(y\)-axes.  Moreover, $f_{q}(\mathbb{D})$ is convex in the \(y\)-direction and $f_{q}(\mathbb{D})$ contains $B_{\frac{1}{2}}(0)$.  
\end{lemma} 
\begin{proof}
    It is clear that $f_{q}$ is holomorphic; it suffices to prove $f_{q}$ is injective. Assume $f_{q}$ is not injective, then there exists $z_1\neq z_2$, such that $f_{q}(z_1)=f_{q}(z_2)$, which is equivalent to $(z_1-z_2)(1+q z_1z_2)=0$. By assumption, $z_1-z_2\neq 0$ and $|1+qz_1z_2|>1-q|z_1z_2|=1-q>0$, which is a contradiction. Hence $f_{q}$ is a univalent function. Moreover, it is direct to verify $f_{q}(z)=\overline{f_{q}(\overline{z})}$ and   $if_{q}(iz)=\overline{i f_{q}(-i\overline{z})}$, we have $f_{q}$ maps the  \(x\)-axis to the \(x\)-axis and the \(y\)-axis to the \(y\)-axis and  $f_{q}(\mathbb{D})$ is symmetric with respect to both the \(x\)- and \(y\)-axes. 
    Now we check the convexity of $f_{q}(\mathbb{D})$ in \(y\)-direction. By Proposition \ref{prop: convex}, it suffices to check $\Re[ f_{q}^{'}(z)(1-z^2)]>0$ for all $z\in \mathbb{D}$. By explicit calculation, we have \begin{equation}
        f_{q}^{'}(z)(1-z^2)= \frac{(1-z^2)(1+qz^2)}{(1-qz^2)^2},
    \end{equation}
    and thus it suffices to check $\Re\Big[(1+qz^{2})(1-z^{2})(1-q\bar z^{2})^{2}\Big]>0$, for all $z\in \mathbb{D}$.
    Let $z=\rho e^{i\theta}$ with $\rho\in[0,1)$ and $c=\cos{2\theta}$, we have 
\begin{align*}
\Re\Big[(1+qz^{2})(1-z^{2})(1-q\bar z^{2})^{2}\Big]
&=
-2q(1-q)\rho^{4}\,c^{2}
+(1+q)\rho^{2}(q^{2}\rho^{4}-1)\,c
+\Big(1+3q(1-q)\rho^{4}-q^{3}\rho^{8}\Big) \\
&=
-2q(1-q)\rho^{4}\cos^{2}(2\theta)
+(1+q)\rho^{2}(q^{2}\rho^{4}-1)\cos(2\theta)
+\Big(1+3q(1-q)\rho^{4}-q^{3}\rho^{8}\Big)\\&:=F(c).
\end{align*} 
As $-2q(1-q)\rho^{4}<0$, $(1+q)\rho^{2}(q^{2}\rho^{4}-1)<0$ and $F(1)=(1-\rho^2)(1-q\rho^2)(1-q^2\rho^4)>0$, we have $F(c)>0$ for all $c\in[-1,1]$.
In the end, $f_{q}$ is univalent, $f_{q}(\partial \mathbb{D})$ is a closed Jordan curve and $0$ is in the bounded component of $\mathbb{C}-f(\partial \mathbb{D})$. As for all $z\in \partial{\mathbb{D}}$, we have $|f_{q}(z)|=\frac{1}{|1-qz^2|}\geq \frac{1}{1+q}\geq\frac{1}{2}$, by the Jordan 
curve theorem, $f_{q}(\mathbb{D})$ contains $B_{\frac{1}{2}}(0)$.
\end{proof}
From now on, for a fixed $q\in (0,1)$, we denote $\Omega_{q}:=f_{q}(\mathbb{D})$. $\Omega_q$ is called the Neumann oval. Now we consider the steady states in the Neumann ovals with constant vorticity, and we have the following result.
\begin{proposition}\label{prop:constant vorticity 1}
    For  $q\in(\sqrt{2}-1,1)$, let $\psi_q$ be the solution to the following equation.\begin{equation}\label{eqn: constant vorticity}
        \begin{aligned}
             &\Delta\psi_{q}=1 \text{ in $\Omega_{q}$,}
             \\& \psi_{q}=0 \text{ in $\partial \Omega_{q}$.}
        \end{aligned}
    \end{equation} 
    Then $\psi_q$ is even both in $x$ and $y$.
    Moreover, $n_{q}=\sqrt{\frac{q^2+2q-1}{q+2q^2-q^3}}$ and $x_{\pm}=(f_{q}(\pm n_{q}),0)$,  then $\{0,x_{\pm}\}$ are the critical points of $\psi_{q}$.
    \end{proposition}

    \begin{proof}
    $\psi_q(x,y)-\psi_q(x,-y)$ is a harmonic function that vanishes at $\partial \Omega_{q}$, hence $\psi_q(x,y)-\psi_q(x,-y)=0$ and $\psi_q$ is even in $y$. Similarly, $\psi_q$ is even in $x$. 
    Now we will explicitly calculate $\psi_q$.
    \begin{lemma}\label{lemma: psi expression}
        Let $\psi_q$ be the solution to the following equation.\begin{equation}
        \begin{aligned}
             &\Delta\psi_{q}=1 \text{ in $\Omega_{q}$,}
             \\& \psi_{q}=0 \text{ in $\partial \Omega_{q}$.}
        \end{aligned}
    \end{equation}
    Then we have \begin{equation}
       \psi_q(x,y)=\frac{y^2}{2}+\Re(\Tilde{\Phi}_{q}\circ f_{q}^{-1}(x+yi)), \text{where $\Tilde{\Phi}_{q}(z)=\frac{z^2}{4(1-qz^2)^2}-\frac{1}{2(1-q^2)(1-qz^2)}$.}
    \end{equation}
    \end{lemma}
    
    \begin{proof}
         Define $\phi_q=\psi_q-\frac{y^2}{2}$, then $\phi$ satisfies \begin{equation}
\begin{aligned}
   &\Delta \phi_q(x,y)=0,\\
 &\phi_{q}(x,y)=\frac{-y^2}{2}, \text{for $(x,y) \in \partial \Omega_{q}.$}
\end{aligned}
\end{equation}
Due to the fact that $\Omega_{q}$ is simply connected, there is a holomorphic function $\Phi_{q}$ such that \begin{equation}
    \phi_q=Re(\Phi_{q}).
\end{equation}
Thus to calculate $\psi_q$, it suffices to calculate $\Phi_{q}$. We now define \begin{equation}
    \Tilde{\Phi}_{q}=\Phi_{q}\circ f_{q}, 
\end{equation}
we have \begin{equation}\label{Riemann}
    \begin{aligned}
    &\Re(\Tilde{\Phi}_{q}(e^{i \theta}))=\frac{-f_2^2(e^{i\theta})}{2}
    &=\frac{-[\sum_{k=0}^{\infty}q^{k}\sin{(2k+1)\theta}]^2}{2}=\sum_{k=0}^{\infty} C_{k} \cos{2k\theta}\end{aligned},
\end{equation}
where \begin{equation}
    \begin{aligned}
    C_k&=\frac{1}{4}\sum_{k_1+k_2=k-1} q^{k_1+k_2}-\frac{1}{4}\sum_{|k_1-k_2|=k} q^{k_1+k_2}\\
    &=\frac{kq^{k-1}}{4}-\frac{q^{k}}{2(1-q^2)}.
    \end{aligned}
\end{equation}
By \eqref{Riemann}, we have \begin{equation}
    \begin{aligned}
     \Tilde{\Phi}_{q}&=\sum_{k=0}^{\infty}(\frac{kq^{k-1}}{4}-\frac{q^{k}}{2(1-q^2)})z^{2k}=\frac{z^2}{4(1-qz^2)^2}-\frac{1}{2(1-q^2)(1-qz^2)}.
    \end{aligned}
\end{equation}
Moreover, we have \begin{equation}\label{eqn:f inverse }
    f_{q}^{-1}(z)=\frac{-1+\sqrt{1+4qz^2}}{2qz},
\end{equation} 
and hence, we get \begin{equation}
    \Phi_{q}(z)=\Tilde{\Phi}_{q}\circ f_{q}^{-1}(z)=\frac{z^2}{4}-\frac{\sqrt{1+4qz^2}+1}{4(1-q^2)}
\end{equation}
 Based on the above calculation, we can calculate $\psi_q$ analytically, \begin{equation}\label{Riemann2}
\begin{aligned}
\psi_q(x,y)&=\frac{y^2}{2}+\Re(\Tilde{\Phi}_{q}\circ f_{q}^{-1}(x+yi)).
\end{aligned}
 \end{equation}
    \end{proof}
 Now that we have an explicit expression of $\psi_q$ in Lemma \ref{lemma: psi expression}, we calculate the critical points of $\psi_q$.  First, as $\Omega_{q}$ is convex in the \(y\)-direction and $\psi_{q}$ is non-positive, by the moving plane method, the critical points of $\psi_q$ have to lie on the  \(x\)-axis. In addition, as $\nabla(\frac{y^2}{2})=0$ on $\{y=0\}$. Hence, the critical points of $\psi_q$ have to be critical points of $\Re(\Tilde{\Phi}_{q}\circ f_{q}^{-1}(x+yi))$. 
 Based on the fact that $f_{q}$ is univalent, it maps the \(x\)-axis to the  \(x\)-axis, it suffices to know the critical points of $\Tilde{\Phi}$ on the  \(x\)-axis. Define $n_{q}=\sqrt{\frac{q^2+2q-1}{q+2q^2-q^3}}$.
By explicit calculations, $0$ and $(\pm n_{q},0)$ are the three critical points of $\Tilde{\Psi}$ on the  \(x\)-axis.
Let $x_{\pm}=(f_{q}(\pm n_q),0)$,  $\{0,x_{\pm}\}$ are the three critical points of $\psi_{q}$. 
    \end{proof}
In the next theorem, we are going to use the formula in Lemma \ref{lemma: psi expression} to show $\psi_{q}$ is a Morse function.
\begin{proposition}\label{prop:constant vorticity 2}
    For almost every $q\in (\sqrt{2}-1,1)$, $0$ is the non-degenerate saddle point for $\psi_{q}$, and $x_{\pm}$ are the non-degenerate local minima.
\end{proposition}
\begin{proof}
    First, the Taylor series for $\sqrt{1+4qz^2}$ converges absolutely for $|z|\leq \frac{1}{2}$, from Lemma \ref{lemma: psi expression}, we have \begin{equation}
    \psi_{q}(x,y)=\frac{y^2}{2}+\Re [\frac{(x+yi)^2}{4}-\frac{\sum_{k=1}^{\infty}{\frac{1}{2}\choose k }(4qz^2)^{k}}{4(1-q^2)}].
\end{equation}
In particular, we have $Hess(\psi_{q})(0)= \begin{bmatrix}
        &\frac{1-q^2-2q}{4(1-q^2)} & 0\\
        &0&   \frac{1-q^2+2q}{4(1-q^2)}
    \end{bmatrix}.$
    It is direct to verify that for all $q\in(\sqrt{2}-1,1)$, $0$ is a strict saddle point.\\
Now we want to prove for a.e. $q\in (\sqrt{2}-1,1)$, $x_{\pm}$ are non-degenerate critical points.
First, we have $1+4qz^2 \neq 0$, for all $z\in \Omega_{q}$. Otherwise, suppose that there exists an $\eta \in \mathbb{D}$ such that  $1+4qf_{q}^2(\eta)=0$, and we have $(1+q\eta^2)^2=0$, which is a contradiction. Hence, we have $\Tilde{\Phi}_{q}(z)$ is jointly analytic in $q$ and $z$. Moreover, by the inverse function theorem, we have $f_{q}^{-1}(z)$ is jointly analytic in $q$ and $z$. As a consequence, we have  $D^2\psi_{q}(x_{+})$ and $D^2\psi_{q}(x_{-})$ are  Puiseux functions in $\sqrt{q-\sqrt{2}+1}$ for all  $q\in [\sqrt{2}-1,1)$. Hence, it suffices to prove $Det(D^2\psi_{q}(x_{+}))$ is not identically zero. Note that $|q-\sqrt{2}+1|\ll 1$, we have $|x_{\pm}|\ll 1$, and hence  again from the absolute convergence for the Taylor expression of $\sqrt{1+4qz^2}$ for $|z|\leq \frac{1}{2}$, we have $Hess(x_{\pm})=\begin{bmatrix}
        &\frac{1-q^2-2q}{4(1-q^2)} & 0\\
        &0&   \frac{1-q^2+2q}{4(1-q^2)}
    \end{bmatrix}+O(|x_{\pm}|^2),$ for $0<q-\sqrt{2}+1\ll 1$.
    As a result, we have $Det(D^2\psi_{q}(x_{\pm}))=\frac{-7+5\sqrt{2}}{4}(q-\sqrt{2}+1)+O((q-\sqrt{2}+1)^2)$. Thus, for almost every $q\in(\sqrt{2}-1,1)$, $x_{\pm}$ are non-degenerate critical points. Based on the maximum principle, there are no local maxima for $\psi_{q}$. In addition, as we have proved $0$ is a strict saddle point, then from the Morse degree theory $x_{\pm}$ have to be strict local minima.
\end{proof}
    \subsubsection{Proof of Proposition \ref{thm:construction}}
    In this section, we are going to finish the proof of Proposition \ref{thm:construction}.
    \begin{proof}
        Due to Proposition \ref{prop:constant vorticity 1} and Proposition \ref{prop:constant vorticity 2}, for almost every $q\in(\sqrt{2}-1,1)$, the solution $\psi_{q}$ to \eqref{eqn: constant vorticity} is a Morse function with three non-degenerate critical points.
      Now for $\lambda \ll1,$ let $\psi_{\lambda}$ be the solution to \eqref{eqn:u def}.
  Denote $\Tilde{\psi}:=\psi^{\lambda}-\psi_{q}$, \eqref{eqn:u def} is equivalent to 
    \begin{equation}\label{eqn:Shauder fixed point}
        \begin{aligned}
            \Tilde{\psi}=\lambda \Delta^{-1}\Tilde{\psi}+\lambda \Delta^{-1} \psi_{q}.
        \end{aligned}
    \end{equation}
    By a similar application of the Schauder fixed point theorem in the proof of Proposition \ref{prop:C1 approximation}, for $\lambda\ll 1$, there exists a solution to \eqref{eqn:Shauder fixed point} with $\|\Tilde{\psi}\|_{L^{\infty}}\leq C$. Similarly, as in the proof of Proposition \ref{prop:smoothness-psi2}, for $\lambda \ll 1$, $\Tilde{\psi}$ is a smooth function, and for all $k\in \mathbb{N}$, there is a constant $C_{k}$, such that $\|\Tilde{\psi}\|_{C^{k}}\leq C_{k}\lambda$. As a result, $\psi^{\lambda}:=\Tilde{\psi}+\psi_{q}$ is a Morse solution to \eqref{eqn:u def}, with three non-degenerate critical points.  
    \end{proof}
    In the end of this section, we finish the proof of Proposition \ref{thm:main} by approximating the steady state in Proposition \ref{thm:main} by smooth steady states satisfying the conditions in Section 2.1.
    \subsection{Conclusion of the proof of Proposition \ref{thm:main}}\label{sec:conclusion}
    We first show that we can approximate general Arnold-stable steady states by smooth Arnold-stable steady states satisfying semilinear elliptic equations with approximately flat nonlinearities. 
    \begin{proposition}\label{prop:F approximate flat}
    Let $\psi_0$ be a solution to $\psi_0=\Delta^{-1}F(\psi_0)$ in a smooth simply connected domain $\Omega$. Assume $F$ is $C^{\lfloor\frac{k}{2}-1\rfloor}$, the operator $\Delta-F^{'}(\psi_0)$ is invertible from $H_{0}^2$ to $L^2$ and $\psi_0$ is a Morse function, then for all $0<\delta\ll 1 $, there is a smooth function $F_{\delta}$ such that the following hold.
    \begin{itemize}
        \item[Property A] There is a solution $\psi_{\delta}$ to $\psi_{\delta}=\Delta^{-1}(F_{\delta}(\psi_{\delta}))$ satisfying $\lim_{\delta\rightarrow 0}\|\psi_{\delta}-\psi_0\|_{C^{\lfloor\frac{k}{2}-1\rfloor}}=0$. 
        \item [Property B]$\psi_{\delta}$ is a smooth Morse function and $F_{\delta}$ is approximately \textbf{$\lfloor \frac{k}{2}\rfloor$-flat} with respect to $\psi_{\delta}$.
        \item [Property C] The operator $\Delta- F_{\delta}^{'}(\psi_{\delta})$ is an invertible map from $H_{0}^2$ to $L^2$.
    \end{itemize}
\end{proposition}

\begin{proof}
We first give the definition of $F_{\delta}$.
Assume $t_1<t_2<\cdots<t_m$ are all the critical values of $\psi_0$, and the range of $\psi_0$ is $[a,b]$, we first define $G$ to be a global extension of  $F^{(\lfloor \frac{k}{2}-2\rfloor)}$.
\begin{align*}
    G(t):=&\left\{\begin{array}{cc}
           F^{(\lfloor \frac{k}{2}-1\rfloor)}(a)(t-a)+ F^{(\lfloor \frac{k}{2}-2\rfloor)}(a),\quad & t\leq a, \\
        F^{(\lfloor \frac{k}{2}-2\rfloor)}(t), &\quad a<t<b, \\
         F^{(\lfloor \frac{k}{2}-1\rfloor)}(b)(t-b)+ F^{(\lfloor \frac{k}{2}-2\rfloor)}(b), &\quad t\geq b.
    \end{array}\right.
\end{align*}
Correspondingly, we now can extend $F$ such that  \begin{equation}
    \frac{d^{\lfloor \frac{k}{2}-2\rfloor}F}{dt^{\lfloor \frac{k}{2}-2\rfloor}}=G,\text{for $t\geq b$, or $t\leq a$.}
\end{equation}
with the boundary condition $ \frac{d^{j}F}{dt^{j}}(a)=F^{(j)}(a)$ and $ \frac{d^{j}F}{dt^{j}}(b)=F^{(j)}(b)$, F is a now a global $C^{\lfloor\frac{k}{2}-1\rfloor}$ function.
    Now for all $\delta>0$, we define an approximation of $G$:
    \begin{align*}
    G_{\delta}^{1}(t):=&\left\{\begin{array}{cc}
           G(t),\quad & min(|t-t_{i}|)\geq \delta \\
         \frac{G(t_{i}+\delta)-G(t_{i}-\delta)}{2\delta}(t-t_i+\delta)+G(t_{i}-\delta), &\quad |t-t_i|\leq \delta
    \end{array}\right.
\end{align*}
Correspondingly, we define $F_{\delta}^{1}$ to be the unique solution to \begin{equation}
    \frac{d^{\lfloor \frac{k}{2}-2\rfloor}F_{\delta}^{1}}{dt^{\lfloor \frac{k}{2}-2\rfloor}}=G_{\delta}^{1},
\end{equation}
with the boundary condition $ \frac{d^{j}F_{\delta}^{1}}{dt^{j}}(a)=F^{(j)}(a)$, for all $0\leq j\leq \lfloor \frac{k}{2}-3\rfloor.$
Now set \(h_\delta:=F_\delta^{1}-F\). 
By construction, \(D^{\left\lfloor k/2\right\rfloor-2}h_\delta\) is supported in a union of intervals of total length \(O(\delta)\), and
\[
\|D^{\left\lfloor k/2\right\rfloor-2}h_\delta\|_{L^\infty}\lesssim \delta .
\]
Moreover, integrating once gives
\begin{equation}\label{eqn:hdifference}
\begin{aligned}
    &\|D^{\left\lfloor k/2\right\rfloor-3}h_\delta\|_{L^\infty}\lesssim \delta^2.
\end{aligned}    
\end{equation}

Therefore, by interpolation,
\[
[D^{\left\lfloor k/2\right\rfloor-3}h_\delta]_{C^{1/2}}
\lesssim
\|D^{\left\lfloor k/2\right\rfloor-3}h_\delta\|_{L^\infty}^{1/2}
\|D^{\left\lfloor k/2\right\rfloor-2}h_\delta\|_{L^\infty}^{1/2}
\lesssim \delta^{3/2}.
\]
Thus
\begin{equation}\label{eqn:close}
    \|F_\delta^{1}-F\|_{C^{\left\lfloor k/2\right\rfloor-3,1/2}[a-1,b+1]}
\lesssim \delta^{3/2}
\end{equation}

Now, let $K$ be a non-negative function from $\mathbb{R}$ to $\mathbb{R}$ with integral $1$ and compact support in $[-1,1]$. We define the rescaled function $K$, $K_{\epsilon}(t)=\frac{K(\frac{t}{\epsilon} )}{\epsilon}$ and use convolution to define an approximation of $F_{\delta}^{1}$:
 \[  F_{\delta}(t):=F_{\delta}^{1}\star K_{\delta ^{2}}(t).\]
 
We have $F_{\delta}$ is smooth, now based on \eqref{eqn:close} and \eqref{eqn:hdifference}, for small $\delta>0$, we have  
\begin{equation}\label{eqn:F approximation close}
    \|F_{\delta}-F\|_{C_{\lfloor\frac{k}{2}-3\rfloor,\frac{1}{2}}[a-1,b+1]}\lesssim \delta^{\frac{3}{2}},
\end{equation}
 \begin{equation}\label{eqn:F approximation close2}
    \|F_{\delta}-F\|_{C_{\lfloor\frac{k}{2}-2\rfloor}[a-1,b+1]}\lesssim \delta,
\end{equation}
and $F^{\lfloor \frac{k}{2}\rfloor}$ vanish in a neighborhood of the critical value of $\psi_0$ with size $\frac{1}{2}\delta$.
As $\Delta-F^{'}(\psi_0)$ is invertible from $H^2$ to $L^2$, $\psi_{\delta}=\Delta^{-1}(F_{\delta}(\psi_{\delta}))$ is then equivalent to $\psi_{\delta}=\mathcal{L}(\psi_{\delta}):=\psi_0+(\Delta-F^{'}(\psi_0))^{-1}[F(\psi_{\delta})-F(\psi_0)-F^{'}(\psi_0)(\psi_{\delta}-\psi_0)]+(\Delta-F^{'}(\psi_0))^{-1}[F_{\delta}(\psi_{\delta})-F(\psi_{\delta})]$.
Now by \eqref{eqn:F approximation close}, elliptic regularity and the explicit form of $\mathcal{L}$, for all $\|f-\psi_0\|_{C^{\lfloor \frac{k}{2}-1\rfloor}}\lesssim \delta$, we have \begin{equation}\label{eqn:Lcompactincreaseregularity}
    \|\mathcal{L}f-\psi_0\|_{C^{\lfloor \frac{k}{2}-1\rfloor,\frac{1}{2}}}\lesssim \delta^{\frac{3}{2}}+\delta \|f-\psi_0\|_{C_{\lfloor \frac{k}{2}-1\rfloor}}.
\end{equation} 
In particular, there exists a constant $C_1$, such that $\{\|f-\psi_0\|_{C^{\lfloor \frac{k}{2}-1\rfloor}}\leq C_1\delta^{\frac{3}{2}}\}$ is invariant under the mapping $\mathcal{L}$. Due to \eqref{eqn:Lcompactincreaseregularity}, $\mathcal{L}$ increases the regularity, it is a compact operator from the $C^{\lfloor \frac{k}{2}-1\rfloor}$ ball around $\psi_0$ to itself. By the Schauder fixed point theorem, there exists a $C^{\lfloor \frac{k}{2}-1\rfloor}$ solution to $\psi_{\delta}=\Delta^{-1}F_{\delta}(\psi_{\delta})$ satisfying $\|\psi_{\delta}-\psi_0\|_{C^{\lfloor \frac{k}{2}-1\rfloor}}\leq C_1\delta^{\frac{3}{2}}$.
By the elliptic regularity, $\psi_{\delta}$ is smooth. Moreover, let $\{z_1,\cdots,z_{n}\}$ be the critical points of $\psi_0$, due to the Morse condition and $\|\psi_{\delta}-\psi_0\|_{C^2}\lesssim \delta^{\frac{3}{2}}$,  the critical points of $\psi_{\delta}$ are $\{d_1,\cdots,d_{n}\}$ with $|d_{i}-z_{i}|\lesssim \delta^{\frac{3}{2}}$. This implies that those critical points are non-degenerate and $|\psi_{\delta}(d_{i})-\psi_0(z_i)|\lesssim \delta^{\frac{3}{2}}\ll \delta$. In particular, we have $F_{\delta}^{(\lfloor\frac{k}{2}\rfloor)}=0$ near the critical values of $\psi_{\delta}$.

In the end, we show if $\delta$ is sufficiently small, then $\Delta- F_{\delta}^{'}(\psi_{\delta})$ is invertible from $H_{0}^2$ to $L^2$. Assume it is not the case, then there exists $\delta_{n}\rightarrow 0$, with $\Delta- F_{\delta}^{'}(\psi_{\delta})$ being not invertible from $H_{0}^2$ to $L^2$. By the Fredholm alternative, in particular, there is $\eta_{n} \in H_{0}^2$, with $\Delta \eta_{n}=F_{\delta}^{'}(\psi_{\delta}) \eta_{n}$, $\|\eta_{n}\|_{L^2}=1$, and we have $\eta_{n}$ is uniformly bounded in $H^1$. In particular, up to a subsequence, $\eta_{n}$ converges weakly to $\eta$ in $H^1$ with $\|\eta\|_{L^2}=1$. By \eqref{eqn:F approximation close2}, as $F_{\delta_{n}}^{'}(\psi_{\delta_{n}})$ converges to $F^{'}(\psi_{0})$ in $L^{\infty}$, $\eta$ is a weak solution to $\Delta \eta=F^{'}(\psi_0)\eta$. By the elliptic regularity theory, $\eta$ is a non-trivial strong solution to $\Delta \eta=F^{'}(\psi_0)\eta$, and it contradicts the invertibility of $\Delta -F^{'}(\psi_0)$ from $H_{0}^2$ to $L^2$.
\end{proof}
Now we are able to finish the proof of Proposition \ref{thm:main}.
\begin{proof}[Proof of Proposition \ref{thm:main}]
  As $F$ is $C^{\lfloor \frac{k}{2}-1\rfloor}$, and the Schr\"odinger operator is invertible, $\psi_0$ is a Morse function,  we apply Proposition \ref{prop:F approximate flat} to construct a smooth Morse steady state $\bar{\psi}$ close to $\psi_0$ in $C^{\lfloor \frac{k}{2}-1\rfloor}$ satisfying $\Delta \bar{\psi}=\bar{F}(\bar{\psi})$ and $\bar{F}$ is approximately \textbf{$\lfloor \frac{k}{2}\rfloor$-flat} with respect to $\bar{\psi}$.
  As we argued in Proposition \ref{lemma:prop2.1}, we can find smooth steady states $\psi_{\epsilon}$ close to $\bar{\psi}$ in $C^{\lfloor\frac{k}{2}-1\rfloor}$ such that $\psi_{\epsilon}$ fails to solve a semilinear elliptic equation. Moreover, there is a $C^2$ neighborhood of $\psi_{\epsilon}$, where all steady states in this neighborhood fail to solve a semilinear elliptic equation.    
\end{proof}
Now we have finished the proof of the first part of Theorem \ref{thm: main theorem}, in the next section, we will finish the second part of Theorem \ref{thm: main theorem}.

\subsection{Dynamical stability of the steady states in Proposition \ref{thm:main}}

In this section, we discuss the dynamical stability of $\psi_{\epsilon}$ we constructed in Proposition \ref{lemma:prop2.1} under the condition that $F^{'}>0$.
First, we note that $\omega_{\epsilon}$ satisfies the Arnold’s stability criterion.
\begin{proposition}\label{prop:stablity}
    Under the setting of Proposition \ref{lemma:prop2.1}, if we in addition, assume $F^{'}>0$, let  $\psi_{\epsilon}$ be the steady states we constructed in Proposition \ref{lemma:prop2.1}. Then, $C_{\epsilon}>\frac{\nabla \Delta \psi_{\epsilon}}{\nabla \psi_{\epsilon}}>c_{\epsilon}$, for positive constants $C_{\epsilon}, c_{\epsilon}$.
\end{proposition}
\begin{proof}
    Due to Proposition \ref{lemma:prop2.1}, we have $\frac{\nabla\Delta \psi_{\epsilon}}{\nabla\psi_{\epsilon}}$ converges to the positive function $F^{'}(\psi_0)$ in $L^{\infty}$ as $\epsilon \rightarrow 0$, in particular, $\frac{\nabla \Delta \psi _{\epsilon}}{\nabla\psi_{\epsilon}}$ is positive for small $\epsilon$. Hence $\psi_{\epsilon}$ satisfies the Arnold's stability criterion.
\end{proof}
Once we know that the steady state $\psi_{\epsilon}$ satisfies the Arnold’s stability criterion, Theorem \ref{thm: main theorem} follows from a classical work of Arnold.
\begin{proposition}\label{prop:linear stablity} 
If the stream function of the steady state $\psi_0$ satisfies the Arnold's stability criterion, and $\psi_0$ is a Morse function, then
$\psi_0$ is linearly stable in $L^2$ of vorticity. 
\end{proposition}
For completeness, we provide the proof below.
\begin{proof}
     Let $\omega_0=\Delta \psi_0$ and $u_0=\nabla^{\perp} \psi_0$, the steady state condition implies that $\nabla^{\perp} \omega_0$ is parallel to $u_0$, and we can define the scalar function $\frac{\nabla^{\perp}\omega_{0}}{  u_{0}}$.
     
     We now prove \begin{equation}\label{eqn:identity casir1}
            u_{0}\cdot \nabla(\frac{\nabla^{\perp}\omega_{0}}{  u_{0}})=0.
     \end{equation}
     First, let $x_0$ be a point such that $  u_{0}(x_0)\neq 0$, then by the inverse function theorem, there is a smooth function $G$ such that $\omega_{0}=G(\psi_{0})$ near $x_0$. As a result, we have \begin{equation}
\begin{aligned}
     &\quad   u_{0}\cdot \nabla(\frac{\nabla^{\perp}\omega_{0}}{  u_{0}})(x_0)\\&=  u_{0}(x_0)\cdot \nabla(G^{'}(\psi_0))(x_0)\\&=  u_{0}(x_0)\cdot G^{''}(\psi_0) \nabla^{\perp} \psi_{0}(x_0)=0.  
\end{aligned}
\end{equation}
    As $x_0$ is an arbitrary point with  $  u_{0}\neq 0$, we have \eqref{eqn:identity casir1} for all points where $  u_{0}\neq 0$. By Proposition \ref{lemma:prop2.1}, $  u_{0}\cdot \nabla(\frac{\nabla^{\perp}\omega_{0}  }{  u_{0}})$ is a  continuous function, and since $  u_{0}\neq 0$ holds almost everywhere, we have \eqref{eqn:identity casir1} for all points.
    Now we consider the linearized Euler equation near $\omega_{0}  $ \begin{equation}
\begin{aligned}
      &\partial_{t} \tilde{\omega}+\tilde{u}\cdot \nabla \omega_{0}+  u_{0}\cdot \nabla \tilde{\omega}=0,\\&
      \tilde{u}=\nabla^{\perp} \Delta ^{-1}\tilde{\omega}.
\end{aligned}
\end{equation}
    Denote $\tilde{\psi}=\Delta^{-1}\tilde{\omega}$, due to \eqref{eqn:identity casir1}, we have  
    \begin{equation}
    \partial_{t}\tilde{\omega} +  u_{0}\cdot \nabla (\tilde{\omega}-\frac{\nabla^{\perp}\omega}{  u_{0}}\tilde{\psi})=0.
\end{equation}
    
    We now introduce the bilinear form $\mathcal{J}(\Tilde{\omega}):=-\frac{1}{2}\int_{\Omega}\Tilde{\omega} \Delta^{-1} \Tilde{\omega} dx+ \int_{\Omega} \frac{  u_{0}}{\nabla^{\perp} \omega_{0}} \Tilde{\omega}^2dx$.  The fact that $\psi_{0}  $ satisfies Arnold's stability criterion implies that  there exists a constant $C>0$ such that \begin{equation}\label{eqn:Casimir identity equivalent}
C\|\tilde{\omega}\|_{L^2}^2\leq  \mathcal{J}(\Tilde{\omega})\leq \frac{1}{C}
\|\tilde{\omega}\|_{L^2}^2.    \end{equation}

Hence, we have \begin{equation}
    \begin{aligned}
      &\frac{d\mathcal{J}(\tilde{\omega}(t,\cdot))}{dt} \quad =2\int_{\Omega}\partial_{t} \tilde{\omega} \frac{  u_{0} }{\nabla^{\perp}\omega_{0}}(\tilde{\omega}-\frac{\nabla^{\perp}\omega_{0}  }{  u_{0}  }\tilde{\psi})dx \\& =-\int_{\Omega} \frac{  u_{0}}{\nabla^{\perp}\omega_{0}  }  u_{0} \cdot\nabla (\tilde{\omega}-\frac{\nabla^{\perp}\omega_{0}  } {  u_{0}  }\tilde{\psi})^2dx\\& \quad =\int_{\Omega}   u_{0}\cdot \nabla \frac{  u_{0}}{\nabla^{\perp}\omega_{0}  }(\tilde{\omega}-\frac{\nabla^{\perp}\omega_{0}  } {  u_{0}  }\tilde{\psi})^2dx.
    \end{aligned}
\end{equation}
 Similarly to \eqref{eqn:identity casir1}, we have $  u_{0} \cdot \nabla \frac{  u_{0} }{\nabla^{\perp}\omega_{0}  }=0$, and hence \begin{equation}
    \frac{d\mathcal{J}(\tilde{\omega}(t,\cdot))}{dt}=0.
\end{equation}
 Then, the linear stability of $\psi_0 $ follows from \eqref{eqn:Casimir identity equivalent}.
\end{proof}
In addition, using the same Casimir, the steady states $\omega_0$ are nonlinearly stable for a fairly long time. 
\begin{proposition}\label{prop:nonlinear stability}
Let $\omega_0$ be the vorticity of a steady state in Proposition \ref{prop:linear stablity}, $\omega$ be the solution to the 2D Euler equation \begin{equation}
    \begin{aligned}
        &\partial_{t}\omega+u\cdot \nabla \omega=0,\\&
        u=\nabla^{\perp}\Delta^{-1}\omega.
    \end{aligned}
\end{equation}
  If $\|\omega(0,\cdot)-\omega_{0}(\cdot)\|_{L^{\infty}}\leq 1$, for all $\alpha \in [0,1)$, there exists a positive constant $C$ such that for all 
 $t\in [0,\frac{1}{\|\omega(0,\cdot)-\omega_{0}(\cdot)\|_{L^{2}}^{\alpha}}]$,
 we have \begin{equation}
    \|\omega(t,\cdot)-\omega_{0}(\cdot)\|_{L^2}\leq C\|\omega(0,\cdot)-\omega_{0}(\cdot)\|_{L^2}.
    \end{equation} 
\end{proposition}
\begin{proof}
    Let $\tilde{\omega}$ be the perturbed vorticity solving \begin{equation}\label{eqn:vorticity eqn perturbed}
    \begin{aligned}
        &\partial_{t} \tilde{\omega}+u_{0}\cdot \nabla \tilde{\omega}+\tilde{u}\cdot \nabla \omega_{0}+\tilde{u}\cdot \nabla \tilde{\omega}=0,\\&
        \tilde{u}=\nabla^{\perp}\Delta^{-1}\tilde{\omega}.
    \end{aligned}
\end{equation} 
We also consider the bilinear form $\mathcal{J}(\Tilde{\omega}):=-\frac{1}{2}\int_{\Omega}\Tilde{\omega} \Delta^{-1} \Tilde{\omega} dx+ \int_{\Omega} \frac{u_{0}}{\nabla^{\perp} \omega_{0}} \Tilde{\omega}^2dx$.

Note the Arnold's stability criterion implies that the critical points of $\psi_0$ and $\omega_0$ coincide. Then, as in the proof of Proposition \ref{prop:linear stablity}, we have \begin{equation}\label{eqn:J inequality}
    \frac{d \mathcal{J}(\Tilde{\omega}(t,\cdot))}{dt}=-\int_{\Omega}(\frac{\nabla\omega_{0}}{|\nabla \omega_{0}|^2}\cdot \nabla )(\frac{u_{0}}{\nabla^{\perp}\omega_{0}})\nabla^{\perp}\omega_{0}\cdot \nabla \tilde{\psi}\frac{\tilde{\omega}^{2}}{2}.
\end{equation}
For the sake of convenience of notation, in this proof, $C$ denotes a class of positive constants independent of $\alpha$, which changes from line to line. $C_{\alpha}$ denotes a class of positive constants that depend on $\alpha$, which also changes from line to line.
By \eqref{eqn:J inequality}, we have \begin{equation}
    \frac{d \mathcal{J}(\Tilde{\omega}(t,\cdot))}{dt}\leq C \|\nabla \Tilde{\psi}\|_{L^{\infty}} \|\tilde{\omega}\|_{L^2}^2. 
\end{equation}
For all $\alpha \in (0,1)$, by the interpolation theorem in Sobolev space, and \eqref{eqn:Casimir identity equivalent},   we have \begin{equation}
    \frac{d \mathcal{J}(\Tilde{\omega}(t,\cdot))}{dt}\leq  C_{\alpha} \|\omega(t,\cdot)\|_{L^{\infty}}^{1-\alpha} \|\tilde{\omega}(t,\cdot)\|_{L^2}^{2+\alpha}. 
\end{equation}
As $\|\omega(t,\cdot)\|_{L^{\infty}}=\|\omega(0,\cdot)\|_{L^\infty}\leq \|(\omega(0,\cdot)-\omega_0(\cdot))\|_{L^\infty}+\|\omega_0(\cdot)\|_{L^\infty}$,
we have \begin{equation}\label{eqn:J functinal boothstrap}
    \frac{d \mathcal{J}(\Tilde{\omega}(t,\cdot))}{dt}\leq C_{\alpha}\mathcal{J}(\Tilde{\omega}(t,\cdot))^{1+\frac{\alpha}{2}}.
\end{equation}
Now due to \eqref{eqn:Casimir identity equivalent}, as $\mathcal{J}(\Tilde{\omega}(0,\cdot))\leq C \|\Tilde{\omega}(0,\cdot)\|_{L^2}^2$, by bootstrapping on \eqref{eqn:J functinal boothstrap}, we have for  $\|\Tilde{\omega}(0,\cdot)\|_{L^2}$ sufficiently small, $\mathcal{J}(\Tilde{\omega}(t,\cdot))\leq 2C \|\Tilde{\omega}(0,\cdot)\|_{L^2}^2$ for all 
$t\in [0,\frac{1}{\|\omega(0,\cdot)-\omega_{0}(\cdot)\|^{\alpha}_{L^{2}}}]$.
Then combined with \eqref{eqn:Casimir identity equivalent}, we have \begin{equation}
    \|\omega(t,\cdot)-\omega_{0}(\cdot)\|_{L^2}\leq C\|\omega(0,\cdot)-\omega_{0}(\cdot)\|_{L^2}.
\end{equation} 
\end{proof}  
\section{Flexibility and rigidity near the cellular flow}\label{sec:cellular}
The cellular flow is an important steady state of the 2D Euler equation on the flat torus. The long-time dynamics of 2D Euler near the cellular flow is an important open question; this has been discussed in the works \cite{lin2004nonlinear,cao2026instability,zhao2024inviscid,brue2024enhanced}. In this section, we discuss the flexibility and rigidity of steady states near the cellular flow.  One of the main purposes of this section is to show that the flexibility phenomenon of
Theorem \ref{thm: main theorem} is not an artifact of the nondegeneracy assumption. The cellular
flow
\[
    \psi_0=\sin x\sin y
\]
is the canonical degenerate example: the kernel of $\Delta+2$ (which is the corresponding Schr\"odinger operator to $\psi_0$) on $L^2(\mathbb T^2)$ is four-dimensional:
\[
\ker(\Delta+2)
=
\operatorname{span}\Bigl\{
\sin x\,\sin y,\ \sin x\,\cos y,\ \cos x\,\sin y,\ \cos x\,\cos y
\Bigr\}.
\] The kernel introduces compatibility conditions on $\psi_2$ and the argument of Section~2 cannot be
applied directly. Nevertheless, in a proper subspace of $L^{2}$, by performing a Lyapunov-Schmidt-type argument, we use the freedom of choice in the right inverse of Hamiltonian to satisfy the compatibility condition. This allows
us to construct smooth Morse steady states near \(\psi_0\) which do not satisfy
any global relation \(\Delta\psi=F(\psi)\).
\subsection{Flexibility near the cellular flow}
Our first goal is to prove Theorem \ref{thm:Flexibility near the cellular flow}. 
\begin{proposition}\label{Prop:cellular flexibility}
    Let $V$ be the class of $C^{1}$ functions that are odd in $x$ and centrally symmetric with respect to  $(\frac{\pi}{2},\frac{\pi}{2})$:
    
    \[V:=\{f|f\in C^{1},   f(x,y)=-f(-x,y),\,f(x,y)=f(\pi-x,\pi-y)\}.\] There exists a family of steady states $\psi_{\epsilon}$ in $V$, such that $\psi_{\epsilon}$ converges to $\psi_0$ in the smooth category as $\epsilon \rightarrow 0$. Moreover, $\psi_{\epsilon}$ fails to solve a semilinear elliptic equation and it is neither odd-odd nor symmetric with respect to either diagonal.
\end{proposition}
\begin{proof}
    To construct Morse steady states $\psi_{\epsilon}$ near $\psi_0$ that break the semilinear elliptic equation structure, we seek $\psi_{\epsilon}:=\psi_0+\epsilon\psi_1+\epsilon^2\psi_2$.
Substituting this ansatz into the steady Euler equation, we impose the
first-order condition
\begin{equation}\label{eq:psi1-constraint}
\nabla^\perp \psi_0\cdot \nabla\bigl( (\Delta+2)\psi_1 \bigr)=0.
\end{equation}
Then the full steady equation reduces to the following equation for the
second-order correction
\begin{equation}\label{eq:psi2-equation}
\nabla^\perp\psi_{\epsilon}\cdot \nabla\bigl( (\Delta+2)\psi_2 \bigr)
=g:=
-(\nabla^\perp\psi_1+\epsilon\nabla^\perp\psi_2)\cdot \nabla\bigl( (\Delta+2)\psi_1 \bigr).
\end{equation}
    We first construct $\psi_1$, which is the major term in breaking the semilinear elliptic equation structure. 
\subsection*{Construction of $\psi_1$}
First, let  $G:\overline{\mathbb{R}^{+}}\rightarrow \mathbb{R}$ be a non-trivial smooth function with  compact support in $[\frac{1}{3},\frac{2}{3}]$ with the requirement that  \begin{equation}
    \int_{(0,\pi)^2} G(\psi_0)\sin{x}\sin{y}dxdy=0.
\end{equation}

Now we define  $\eta$ such that,  
$\eta := G(\psi_0)$ on $(0,\pi)^2$, $\eta:=-\eta(-x,y)$  on $(-\pi,0)\times (0,\pi)$  and  $\eta:=0$ in the remaining region.
  It is clear that $\eta$ is a smooth function in $V$. Due to the odd symmetry in $x$, $\eta$ is orthogonal to $\cos{x}\sin{y}$ and $\cos{x}\cos{y}$. As $\eta$ is centrally symmetric with respect to $(\frac{\pi}{2},\frac{\pi}{2})$, we have $\eta$ is orthogonal to $\sin{x}\cos{y}$.
  Moreover, we have \[\int_{\mathbb{T}^{2}}\eta \sin{x}\sin{y}dxdy=2\int_{(0,\pi)^2}G(\psi_0)\sin{x}\sin{y}dxdy=0,\]
  and thus $\eta$ is orthogonal to $\sin{x}\sin{y}$. 
  Now we can correspondingly define $\psi_1:=(\Delta+2)^{-1}\eta$ and $\psi_1$ lies in $V$. 
We now define $\psi_2$.
\subsection*{Construction of $\psi_2$}
Assume that $\|\psi_2\|_{V}\ll \frac{1}{\sqrt{\epsilon}}$, for sufficiently small $\epsilon$, the support of $\eta$ is away from the critical level set of $\psi_{\epsilon}$. 
We define the nonlinear operator 
 \begin{equation}\label{eqn:def phiepsilon}
     \begin{aligned}
         \Phi_\epsilon(\psi_2)
:=
L_{\psi_\epsilon}^{-1}\mathcal{R}_{\epsilon},
     \end{aligned}
 \end{equation}
where 
\begin{equation}\label{eqn:def reminderepsilon}
R_\epsilon
:=
-(\nabla^{\perp}\psi_1+\epsilon \nabla^{\perp}\psi_2)\cdot\nabla\eta.
\end{equation}
 We now show  $\Phi_{\epsilon}(\psi_2)$ is  odd in $x$ and centrally symmetric with respect to $(\frac{\pi}{2},\frac{\pi}{2})$.  By \eqref{eqn:defH-1}, as the support of $\mathcal{R}_{\epsilon}$ is away from the critical level set of $\psi_{\epsilon}$, $\mathcal{L}^{-1}_{\psi_{\epsilon}}\mathcal{R}_{\epsilon}=0$ near the critical level set of $\psi_{\epsilon}$. It suffices to consider the points away from the critical level set of $\psi_{\epsilon}$.
For all $x\in \mathbb{T}^2$, away from the critical level set of $\psi_{\epsilon}$, let $\tilde{\Gamma}(x,t)$ to be the rescaled Hamiltonian flow defined in \eqref{eqn:rescaled H flow} with the Hamiltonian $\psi_{\epsilon}$.
 $\Tilde {\Gamma}(x,\cdot):[0,1]\rightarrow \mathbb{T}^2$ gives a parametrization for the level set of $\psi_{\epsilon}$ crossing $x$. As $\psi_{\epsilon}$ is centrally symmetric with respect to $(\frac{\pi}{2},\frac{\pi}{2})$, $\Tilde {\Gamma}(x,s)$  and $\Tilde {\Gamma}(x,s+\frac{1}{2})$ are centrally symmetric with respect to $(\frac{\pi}{2},\frac{\pi}{2})$.
 Moreover, as $\mathcal{R}_{\epsilon}$ is centrally symmetric with respect to $(\frac{\pi}{2},\frac{\pi}{2})$, we have $\mathcal{R}_{\epsilon}(\tilde{\Gamma}(x,s))=\mathcal{R}_{\epsilon}(\tilde{\Gamma}(x,s+\frac{1}{2}))$, by the formula \eqref{eqn:defH-1} and \eqref{eqn:defM}, 
 \begin{equation}
     \mathcal{L}_{\psi_{\epsilon}}^{-1}\mathcal{R}_{\epsilon}(\tilde{\Gamma}(x,0))= \mathcal{L}_{\psi_{\epsilon}}^{-1}\mathcal{R}_{\epsilon}(\tilde{\Gamma}(x,\frac{1}{2})).
 \end{equation}
Now as $x$ is arbitrary and $\tilde{\Gamma}(x,\frac{1}{2})$ is centrally symmetric to $x$ with respect to $(\frac{\pi}{2},\frac{\pi}{2})$, we now have $\mathcal{L}_{\psi_{\epsilon}}^{-1}\mathcal{R}_{\epsilon}$ is  centrally symmetric with respect to $(\frac{\pi}{2},\frac{\pi}{2})$.
By similar arguments as above, $\mathcal{L}_{\psi_{\epsilon}}^{-1}\mathcal{R}_{\epsilon}$ is odd in $x$, and hence $\Psi_{\epsilon}$ preserves the odd symmetry in $x$ and central symmetry with respect to $(\frac{\pi}{2},\frac{\pi}{2})$.
Now we define a parameter $A$ depending on $\psi_2$, with \begin{equation}
    A=-\dfrac{\int_{\mathbb{T}^2}\Phi_{\epsilon}(\psi_{2})\sin{x}\sin{y}dxdy}{\int_{\mathbb{T}^2}\psi_\epsilon \sin{x}\sin{y}dxdy}.
\end{equation}
$A$ is the parameter such that $\Phi_{\epsilon}(\psi_2)+A\psi_{\epsilon}$ is orthogonal to $\sin{x}\sin{y}$.
Now the mapping $\Psi_{\epsilon}(\psi_2):=(\Delta+2)^{-1}(\Phi_{\epsilon}(\psi_2)+A\psi_{\epsilon})$ is well-defined and also preserves the odd symmetry in $x$ and central symmetry with respect to $(\frac{\pi}{2},\frac{\pi}{2})$.
 We now seek the solution to \eqref{eq:psi2-equation} that solves \begin{equation}\label{eqn:psi2-equation}
    \psi_{2}=\Psi_{\epsilon}(\psi_2).
\end{equation}

 By elliptic regularity and Proposition \ref{lemma:solve transport},  for $|\psi_2|\ll \frac{1}{\sqrt{\epsilon}}$ and sufficiently small $\epsilon$, similarly as in the proof of Proposition \ref{lemma:prop2.1},
we have \[|\Psi_{\epsilon }(\psi_2)|_{C^{1,\frac{1}{2}}(\mathbb{T}^2)}\lesssim_{\psi_1}1+\epsilon \|\psi_2\|_{C^{1}}.\]
In particular,  for small $\epsilon$, there is a $C$, such that the set $K:=\{f|f\in V, \|f\|_{C^{1}}\leq C\}$ is an invariant set for the map $\Psi_{\epsilon}$. Now as the map $\Psi_{\epsilon}$ is a compact continuous map from $K$ to $K$, by the Schauder fixed point theorem, there is a fixed point $\psi_2$ in $K$.\\ The steady state $\psi_{\epsilon}=\psi_0+\epsilon\psi_1+\epsilon^2\psi_2$ will then be a steady state in $V$ close to $\psi_0$ in $C^{1}$, similarly as in Lemma \ref{prop:smoothness-psi2}, we have uniform bounds on $\psi_2$ in the smooth category and hence we have $\psi_{\epsilon}$ converges to $\psi_0$ in the smooth category as $\epsilon \rightarrow 0$.  Now we finish the proof by proving the various topological properties of $\psi_{\epsilon}$.
\subsection*{The topological behavior of $\psi_{\epsilon}$.}
First, we show $\psi_{\epsilon}$ does not satisfy any global semilinear elliptic equation $\Delta\psi_{\epsilon}=F(\psi_{\epsilon})$ on
$\mathbb T^2$ for any function $F$. 
Let  \(t_0\in (1/3,2/3)\) such that \(G(t_0)\neq 0\), and choose
\(p_+^0\in (0,\pi)^2\) with
\[
\psi_0(p_+^0)=t_0.
\]
Let \(p_-^0:=-p_+^0\in (-\pi,0)^2\). Then
\[
\psi_0(p_-^0)=\psi_0(p_+^0)=t_0,
\]
while
\[
\eta(p_+^0)=G(t_0)\neq 0,
\qquad
\eta(p_-^0)=0.
\]
Since \(\psi_\epsilon=\psi_0+O(\epsilon)\) in \(C^1\), the implicit function theorem implies that, for \(\epsilon\) sufficiently small, there exist points
\(p_+\) near \(p_+^0\) and \(p_-\) near \(p_-^0\) such that
\[
\psi_\epsilon(p_+)=\psi_\epsilon(p_-)=t_0.
\]
On the other hand,
\[
\Delta\psi_\epsilon
=
-2\psi_\epsilon+\epsilon\eta+O(\epsilon^2),
\]
and therefore
\[
\Delta\psi_\epsilon(p_+)-\Delta\psi_\epsilon(p_-)
=
\epsilon\bigl(\eta(p_+)-\eta(p_-)\bigr)+O(\epsilon^2)
=
\epsilon G(t_0)+O(\epsilon^2)\neq 0.
\]
Thus \(\Delta\psi_\epsilon\) cannot be a single-valued function of \(\psi_\epsilon\).
Similarly, by comparing $\psi_{\epsilon}$ in $(0,\pi)^2$ and $(0,\pi)\times (-\pi,0)$, $\psi_{\epsilon}$ is not odd-odd, by comparing $\psi_{\epsilon}$ in $(-\pi,0)\times (0,\pi)$ and $(0,\pi)\times (-\pi,0)$, it is not even symmetric with respect to $\{y=x\}$. Similarly, by comparing $\psi_{\epsilon}$ in $(0,\pi)\times (0,\pi)$ and $(-\pi,0)\times (-\pi,0)$, it is not even symmetric with respect to $\{y=-x\}$.
\end{proof}
The preceding construction uses a specific choice of elements in the kernel to break the semilinear elliptic structure.  We note, by contrast, in the odd-odd class, when  these directions are absent, the semilinear elliptic structure is preserved. In the next section, we discuss the rigidity of steady states near the cellular flow and complete the proof of Theorem \ref{thm:Flexibility near the cellular flow}.
\subsection{Rigidity in the odd-odd class}
In this section, we show that in the odd-odd symmetric class, if the steady state is close to the cellular flow, then the stream function solves a semilinear elliptic equation and is even symmetric to both diagonals.
\begin{proposition}\label{prop:celluar symmetry}
Let $\psi_{\epsilon}$ be the stream function of an odd--odd symmetric steady state on $\mathbb T^2$.
Let $\kappa\in(0,1)$, if
\[
\|\psi_{\epsilon}-\psi_0\|_{H^{3+\kappa}(\mathbb T^2)}\leq \epsilon,
\]
then  for sufficiently small $\epsilon$, $\psi_{\epsilon}$ is even symmetric with respect to both diagonals $\{y=x\}$ and $\{y=-x\}$ of
$[-\pi,\pi]^2$. Moreover, there exists a function $F_{\epsilon}$ such that $\psi_{\epsilon}$ is a solution to $\Delta \psi_{\epsilon}=F_{\epsilon}(\psi_{\epsilon})$ in $\mathbb{T}^2$.
\end{proposition}

\medskip
To prove Proposition~\ref{prop:celluar symmetry}, we need the following estimate for the period
function of the Hamiltonian flow.

\begin{proposition}[Period estimates near the cellular flow]\label{prop: Hamiltonian celluar period}
Let $T_{\psi_0+\hat{\psi}}(x,y)$ be the period of the Hamiltonian trajectory of the flow
\begin{equation}\label{eq:Ham-flow}
\begin{aligned}
\frac{dX}{dt} &= \nabla^\perp(\psi_0+\hat{\psi})(X),\\
X(0) &= (x,y).
\end{aligned}
\end{equation}
There exists a universal constant $C>0$ and a constant $\delta_0>0$ such that the following holds.
If  $\hat{\psi}$ is odd--odd symmetric and even symmetric with respect to both diagonals, and satisfies
\[
\|\hat{\psi}\|_{H^{3+\kappa}(\mathbb T^2)}\le \delta_0,
\]
then
\begin{equation}\label{eqn:Hardy1}
\int_{\mathbb T^2}\frac{f^2}{T_{\psi_0+\hat{\psi}}^2}\,dx\,dy
\le C \int_{\mathbb T^2}\bigl|\nabla^\perp(\psi_0+\hat{\psi})\cdot \nabla f\bigr|^2\,dx\,dy,
\end{equation}
for every odd--odd function $f$ that is odd symmetric with respect to $\{y=x\}$.
Moreover, for every $p\in(1,\infty)$ there exists $C_p>0$ (depending only on $p$) such that
\begin{equation}\label{eqn:period Holder}
\|T_{\psi_0+\hat{\psi}}\|_{L^p(\mathbb T^2)} \le C_p.
\end{equation}
\end{proposition}

\begin{proof}
Since $f$, $\psi_0$, and $\hat{\psi}$ are odd--odd, it suffices to work on the first quadrant
\[
\mathbb T^2_{+}:=\mathbb T^2\cap\{x>0,\ y>0\}.
\]

\medskip
We first identify the critical points of $\psi_0+\hat \psi$.
Because $\psi_0$ and $\hat \psi$ are even symmetric with respect to both diagonals, $2\pi$--periodic in $x$ and $y$ and odd-odd symmetric with respect to the origin,
the points
\[
(0,0),\ \Bigl(\frac{\pi}{2},\frac{\pi}{2}\Bigr),\ (\pi,0),\ (0,\pi),\ (\pi,\pi)
\]
are critical points of $\psi_0+\hat \psi$ in $\mathbb T^2_{+}$. By Sobolev embedding,
$\|\hat \psi\|_{C^{2}}\lesssim \|\hat \psi\|_{H^{3+\kappa}}$; hence, for
$\delta_0$ sufficiently small, these critical points remain non-degenerate and no other critical points
appear in $\mathbb T^2_{+}$.

\medskip
Let $\psi:=\psi_0+\hat \psi$. Now we prove \eqref{eqn:Hardy1} via coarea formula and Poincar\'e inequality on level sets of $\psi$.
 Applying the coarea formula to $\psi$ on $\mathbb T^2_{+}$ yields
\begin{equation}\label{eqn:co1}
\int_{\mathbb T^2_{+}} \frac{f^2}{T_{\psi}^2}\,dx\,dy
=
\int_{\mathbb R}\int_{\{\psi=t\}\cap \mathbb T^2_{+}}
\frac{f^2}{T_{\psi}^2}\,\frac{1}{|\nabla\psi|}\,d\mu\,dt,
\end{equation}
and
\begin{equation}\label{eqn:co2}
\int_{\mathbb T^2_{+}} |\nabla^\perp \psi\cdot \nabla f|^2\,dx\,dy
=
\int_{\mathbb R}\int_{\{\psi=t\}\cap \mathbb T^2_{+}}
|\nabla^\perp \psi\cdot \nabla f|^2\,\frac{1}{|\nabla\psi|}\,d\mu\,dt,
\end{equation}
where $d\mu$ denotes the Lebesgue measure in the level sets.

Fix $t$ such that $\{\psi=t\}\cap\mathbb T^2_{+}$ is a smooth closed curve and
let $x_t$ be  the intersection of the curve with the diagonal segment $\{0<y=x<\frac{\pi}{2}\}$, we have the \emph{rescaled} Hamiltonian
parametrization of that level set
$\Gamma_{t}:[0,1]\to \{\psi=t\}\cap\mathbb T^2_{+}$. 
\begin{equation}\label{eqn:rescaled-flow}
\begin{aligned}
\frac{d\Gamma_t}{ds} &= T_{\psi}(x_t)\,\nabla^\perp \psi(\Gamma_t(s)),\\
\Gamma_t(0) &= x_t.
\end{aligned}
\end{equation}
Due to the even symmetry of $\psi$ to the diagonal $\{y=x\}$, we have $\Gamma_{t}(s)$ and $\Gamma_{t}(1-s)$ are even symmetric to the diagonal $\{y=x\}$. 
Along the curve $\Gamma_{t}$, set $\hat f(s):=f(\Gamma_t(s))$. Since $f$ is odd with respect to $\{y=x\}$,
we have $\hat f(0)=0$, $\hat f$ is odd with respect to $\frac{1}{2}$, and hence it has zero
average over $[0,1]$. The Poincar\'e inequality gives
\[
\int_0^1 |\hat f(s)|^2\,ds \le C \int_0^1 |\hat f'(s)|^2\,ds.
\]
Using $\hat f'(s)=T_{\psi}(x_t)\,(\nabla^\perp\psi\cdot\nabla f)(\Gamma_t(s))$ and the identity
$d\mu = |\nabla\psi|\,T_{\psi}(x_t)\,ds$, this becomes
\begin{equation}\label{eqn:poincare}
\int_{\{\psi=t\}\cap\mathbb T^2_{+}}
\frac{f^2}{T_{\psi}^2}\,\frac{1}{|\nabla\psi|}\,d\mu
\le
C \int_{\{\psi=t\}\cap\mathbb T^2_{+}}
|\nabla^\perp\psi\cdot\nabla f|^2\,\frac{1}{|\nabla\psi|}\,d\mu.
\end{equation}
Integrating \eqref{eqn:poincare}  with respect to $t$ over $\mathbb{R}$  and using \eqref{eqn:co1}--\eqref{eqn:co2}, as $\psi_0,\hat \psi$ and $f$ are odd-odd, we have  
\eqref{eqn:Hardy1}.

\medskip
In the end, we will derive \eqref{eqn:period Holder} based on the behavior of $\psi$.
By symmetry, it suffices to work on
\[
\Omega_0:=\{(x,y)\in(0,\pi)^2:\ 0<x<y<\pi-x\}.
\]
Since $\psi=\psi_0+\hat \psi$ and $\|\hat \psi\|_{C^{2,\alpha}}\ll 1$, there exists $C>1$ such that on $\Omega_0$,
\begin{equation}\label{eqn:psi bound 0}
\frac{1}{C}\psi_0 \le \psi \le C\psi_0,
\qquad
\frac{1}{C}\partial_x\psi_0 \le \partial_x\psi \le C\partial_x\psi_0.
\end{equation}
Fix $\delta\in(0,1)$ and decompose $\Omega_0=\Omega_1\cup\Omega_2\cup\Omega_3$ where
\[
\Omega_1:=\Omega_0\cap\{x<\delta\},\qquad
\Omega_2:=\Omega_0\cap B_\delta\Bigl(\Bigl(\frac{\pi}{2},\frac{\pi}{2}\Bigr)\Bigr),\qquad
\Omega_3:=\Omega_0\setminus(\Omega_1\cup\Omega_2).
\]
Below is an illustrative figure for the domain division in $\Omega_0$.
 \begin{figure}[h!]
 \caption{Division of the domain}
    \centering
    \includegraphics[width=0.25\linewidth]{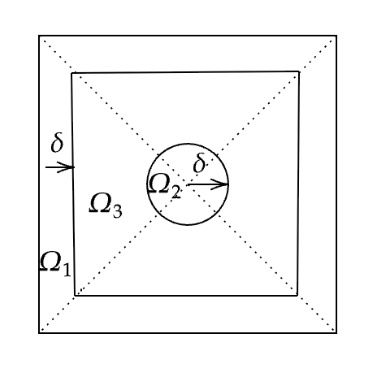}
\end{figure}
For $\psi_0(x,y)=\sin x\sin y$, we have the following bounds (with constants depending on $\delta$):
on $\Omega_1$,
\begin{equation}\label{eqn:psi bound 1}
\psi_0(x,y)\sim xy,\qquad \partial_x\psi_0(x,y)\gtrsim y,
\end{equation}
on $\Omega_2$,
\begin{equation}\label{eqn:psi bound 2}
1-\psi_0(x,y)\sim \Bigl(\frac{\pi}{2}-x\Bigr)^2+\Bigl(\frac{\pi}{2}-y\Bigr)^2,
\qquad
\partial_x\psi_0(x,y)\gtrsim \Bigl|\frac{\pi}{2}-x\Bigr|,
\end{equation}
and on $\Omega_3$,
\begin{equation}\label{eqn:psi bound 3}
\partial_x\psi_0(x,y)\gtrsim 1.
\end{equation}
Combining \eqref{eqn:psi bound 0} with \eqref{eqn:psi bound 1}--\eqref{eqn:psi bound 3} gives analogous
lower bounds for $\partial_x\psi$ and comparable bounds for $\psi$ in the same regions.
From these bounds, we obtain:
\begin{equation}\label{eqn:T bound 1}
T_{\psi}\lesssim_{\delta} 1 \quad \text{on }\Omega_2\cup\Omega_3,
\end{equation}
and
\begin{equation}\label{eqn:T bound 2}
T_{\psi}\lesssim_{\delta} 1+|\log |\psi_0|| \quad \text{on }\Omega_1.
\end{equation}
Since $|\log |\psi_0||\in L^p(\Omega_1)$ for every $p<\infty$, \eqref{eqn:period Holder} follows from
\eqref{eqn:T bound 1}--\eqref{eqn:T bound 2} and the various symmetry assumptions on $\psi_0$ and $\hat \psi$.
\end{proof}
\medskip
We now prove Proposition~\ref{prop:celluar symmetry}.

\begin{proof}[Proof of Proposition~\ref{prop:celluar symmetry}]
Without loss of generality, by scaling, we assume $\langle\psi_{\epsilon}-\psi_0,\psi_0\rangle_{L^2(\mathbb T^2)}=0$.
Among odd--odd functions, we have the following fact:
\begin{quote}
If $H$ is odd--odd on $\mathbb T^2$, then $H$ is even (respectively odd) symmetric with respect to
$\{y=x\}$ if and only if it is even (respectively odd) symmetric with respect to $\{y=-x\}$.
\end{quote}
Based on the above fact, we now decompose $\psi-\psi_0$ into parts that are respectively  even-even symmetric and odd-odd symmetric to the diagonals.
\begin{equation}\label{eqn:decompose odd even}
    \psi_{\epsilon}-\psi_0 = \psi_{ee}+\psi_{oo},
\end{equation}
where $\psi_{ee}$ is even with respect to both diagonals and $\psi_{oo}$ is odd with respect to both
diagonals. Since $\psi_{\epsilon}$ is a steady Euler stream function,
\[
\nabla^\perp\psi_{\epsilon}\cdot\nabla\Delta\psi_{\epsilon}=0.
\]
Using $\Delta\psi_0=-2\psi_0$ and the decomposition \eqref{eqn:decompose odd even}, we have 
\begin{equation}\label{eq:oo-eqn}
\nabla^\perp(\psi_0+\psi_{ee})\cdot \nabla (\Delta+2)\psi_{oo}
=
2\nabla^\perp\psi_{ee}\cdot \nabla \psi_{oo}
-\nabla^\perp\psi_{oo}\cdot \nabla \Delta \psi_{ee}.
\end{equation}
Fix any $q\in(\frac{2}{1+\kappa},2)$ and set $p:=\frac{2q}{2-q}$ so that $\frac1q=\frac12+\frac1p$.
As $\psi_{oo}$ is orthogonal to the kernel of $(\Delta+2)$, by Applying Proposition~\ref{prop: Hamiltonian celluar period} with $f=(\Delta+2)\psi_{oo}$ and 
$\hat{\psi}=\psi_{ee}$, and then using H\"older's inequality and Sobolev embedding
$H^{3+\kappa}\subset W^{3,\frac{2}{1-\kappa}}$, we get
\begin{equation}\label{eq:psi_oo_Lq}
\begin{aligned}
\|\psi_{oo}\|_{W^{1,p}}
&\le C\,\|(\Delta+2)\psi_{oo}\|_{L^{q}}\\
&\le C\left\|\frac{(\Delta+2)\psi_{oo}}{T_{\psi_0+\psi_{ee}}}\right\|_{L^{2}}
\,
\|T_{\psi_0+\psi_{ee}}\|_{L^{p}} \\
&\le
C\,\|T_{\psi_0+\psi_{ee}}\|_{L^{p}}\,
\bigl\|\nabla^{\perp}(\psi_0+\psi_{ee})\cdot \nabla (\Delta +2)\psi_{oo}\bigr\|_{L^2}\\
&\le
CC_{\kappa}\,C_{p}\,
\|\psi_{ee}\|_{H^{3+\kappa}}\,
\|\psi_{oo}\|_{W^{1,p}},
\end{aligned}
\end{equation}
where in the last line we used \eqref{eq:oo-eqn} and that each term on the right-hand side of
\eqref{eq:oo-eqn} is bounded in $L^2$ by $C_\kappa\|\psi_{ee}\|_{H^{3+\kappa}}\|\psi_{oo}\|_{W^{1,p}}$.

Therefore, if $\|\psi_{\epsilon}-\psi_0\|_{H^{3+\kappa}}$ is sufficiently small so that
\[
CC_{\kappa}C_{p}\,\|\psi_{ee}\|_{H^{3+\kappa}}<1,
\]
 \eqref{eq:psi_oo_Lq} implies $\psi_{oo}=0$. Consequently, $\psi_{\epsilon}=\psi_0+\psi_{ee}$ is even symmetric with respect to
both diagonals. This proves the even symmetry of $\psi_{\epsilon}$ to the diagonals.
Now as argued in the proof of Proposition \ref{prop: Hamiltonian celluar period}, the critical points of $\psi_{\epsilon}$ in $\mathbb{T}_{+}^2$ are $\{(0,0),(0,\pi),(\pi,0),(\frac{\pi}{2},\frac{\pi}{2}),(\pi,\pi)\}$. As $\psi_{\epsilon}$ is non-negative in $\mathbb{T}_{+}^2$ with $\psi_{\epsilon}$ and $\Delta \psi_{\epsilon}$ both vanish at $\partial \mathbb{T}_{+}^2$, and the steady state equation implies that $\Delta \psi_{\epsilon}$ is constant along the connected component of the level set of $\psi_{\epsilon}$, there exists a function $F:\mathbb{R}^{+}\rightarrow \mathbb{R}$ with $F(0)=0$, such that $\Delta \psi_{\epsilon}=F(\psi_{\epsilon})$ in $\mathbb{T}_{+}^2$. Now as $\psi_{\epsilon}$ is odd-odd symmetric with respect to the origin, by extending $F$ to be an odd function, we have $\Delta \psi_{\epsilon}= F(\psi_{\epsilon})$ in $\mathbb{T}^2$.
\end{proof}
\begin{remark}
It can be proved, similarly to Proposition~\ref{prop:celluar symmetry}, that if a steady state is odd in \(x\) and sufficiently close to the cellular flow in \(H^{3+\kappa}\), for some \(\kappa>0\), then, up to a translation in the \(y\)-direction, it must be centrally symmetric with respect to \(\left(\frac{\pi}{2},\frac{\pi}{2}\right)\). Thus the choice of the space \(V\) in Proposition~\ref{Prop:cellular flexibility} is not merely technical: under the odd symmetry in \(x\), it is precisely the space in which steady states near the cellular flow live. For a related phenomenon in the context of \(V\)-states, see~\cite{huang2025rigidity}.
\end{remark}
\appendix
\section{Facts from complex analysis}
In this section, we list facts from complex analysis that we will use in the paper.
\begin{proposition}[in \cite{robertson1936theory}]\label{prop: convex}
    Let $f$ be a univalent function from $\mathbb{D}$ to $\mathbb{C}$. If $f$ satisfies $\Re(f^{'}(z)(1-z^2))>0$, for all $z\in \mathbb{D}$, then $f(\mathbb{D})$ is convex in $y$.
\end{proposition}
\renewcommand{\baselinestretch}{1.25}
\section{Facts about action-angle coordinates}
In this section, we list and prove the facts for the action-angle coordinates of the Hamiltonian flow.
\begin{proposition}[Regularity of the period function]\label{prop:regularity period}
Let $\Omega\subset \mathbb{R}^2$ be compact and let $H:\Omega\to\mathbb{R}$ be a Hamiltonian. 
For $x\in\Omega$, let $T_H(x)$ denote the minimal period of the Hamiltonian flow generated by $H$. Fix $a>0$ and define
\[
\mathcal{R}_{a}
:=\Bigl\{x\in\Omega:\ |\nabla H(y)|>a \ \text{for every } y\in\Omega \text{ with } H(y)=H(x)\Bigr\}.
\]
Then:
\begin{itemize}
    \item[(1)] If $H\in C^{1,1}(\Omega)$, for every
    $\tilde H$ with $\|H-\tilde H\|_{C^{1,1}(\Omega)}\le \frac{a}{4}$, there is a constant
    $C\bigl(a,\|H\|_{C^{1,1}(\Omega)}\bigr)$ for which
    \[
    \|T_{\tilde H}\|_{C^{0,1}(\mathcal{R}_{2a})}\le C\bigl(a,\|H\|_{C^{1,1}(\Omega)}\bigr).
    \]
    \item[(2)] Let $k\ge 2$ and $H\in C^{k,\alpha}(\Omega)$, then for every
    $\tilde H$ with $\|H-\tilde H\|_{C^{k,\alpha}(\Omega)}\le \frac{a}{4}$, there is a constant
    $C\bigl(a,\|H\|_{C^{k,\alpha}(\Omega)}\bigr)$ for which
    \[
    \|T_{\tilde H}\|_{C^{k-1,\alpha}(\mathcal{R}_{2a})}\le C\bigl(a,\|H\|_{C^{k,\alpha}(\Omega)}\bigr).
    \]
\end{itemize}
\end{proposition}
\begin{proof}
For the sake of convenience of notation, $C(a,\|H\|_{C^{1,1}})$ denotes a class of positive constants depending on $a$ and $\|H\|_{C^{1,1}}$, and it may change from line to line.
Write
\[
\partial\mathcal{R}_{a}=\bigcup_{i=1}^{n}\Gamma^{i},
\]
where the $\Gamma^{i}$ are pairwise disjoint closed $C^{1,1}$ curves. Fix $x\in \mathcal{R}_{2a}$ and let
$\Gamma^{\tilde H}_{x}$
denote the connected component of the level set $\{\tilde H=\tilde H(x)\}$ passing through $x$.
Together with a (possibly empty) subcollection of boundary components
$\Gamma^{i_1},\dots,\Gamma^{i_j}\subset \partial\mathcal{R}_a$,
the curve $\Gamma^{\tilde H}_{x}$ encloses a region $\tilde\Omega_x\subset \mathcal{R}_a$.

Let
\[
F:=\frac{\nabla \tilde H}{\|\nabla \tilde H\|^{2}}.
\]
 By the fact that $F\cdot n = |\nabla\tilde H|^{-1}$ on
each level curve of $\tilde H$, we apply the divergence theorem to $\tilde\Omega_x$ and obtain
\begin{equation}\label{eq:stokes-period}
T_{\tilde H}(x)
=\int_{\Gamma^{\tilde H}_{x}} F\cdot n\,ds
=\sum_{m=1}^{j}\int_{\Gamma^{i_m}} F\cdot n\,ds
-\int_{\tilde\Omega_x}\nabla\cdot F\,dx\,dy .
\end{equation}
As $\|H-\tilde H\|_{C^{1,1}}\le a/4$, we have
$|\nabla \tilde H|\ge a/2$ on $\mathcal{R}_a$. As a result, we have  $|F|\lesssim a^{-1}$ and $|\nabla F|\lesssim a^{-3}\| \tilde H\|_{C^{1,1}}$.
Using \eqref{eq:stokes-period} and $\tilde\Omega_x \subseteq \mathcal{R}_a$, we obtain the uniform bound of $T_{\tilde H}$ for all $x\in \mathcal{R}_{2a}$.
\begin{equation}\label{eq:TLinfty-bound}
|T_{\tilde H}(x)|
\lesssim \|H\|_{C^{1,1}}\Bigl(\frac1a\sum_{i=1}^{n}|\Gamma^{i}|+\frac{1}{a^3}\,|\mathcal{R}_{a}|\Bigr)\leq C(a,\|H\|_{C^{1,1}}).
\end{equation}

\medskip

Next, we prove the uniform Lipschitz dependence of $T_{\tilde H}$ in $\mathcal{R}_{2a}$. Choose $\epsilon_{1}>0$, such that for every
$x\in \mathcal{R}_{2a}$ the set
\[
\{y\in\mathbb{R}^{2}:\ |\tilde H(y)-\tilde H(x)|<\epsilon_{1}\}
\]
has a connected component $U\subset \mathcal{R}_{a}$ and $U$ is an open tubular neighborhood of $\Gamma^{\tilde H}_{x}$.
Set $h_{0}:=\tilde H(x)$, for $h\in (h_{0}-\epsilon_{1},h_{0}+\epsilon_{1})$,
let
\[
\Gamma_{h}:=\{ \tilde H=h\}\cap U,
\]
 $\{\Gamma_{h}\}$ is a family of $C^{1,1}$ closed curves foliating $U$.

To parameterize this foliation, fix a $C^{1,1}$ parametrization $z(\theta)$ of $\Gamma_{h_{0}}$,
$\theta\in[0,\ell_{0}]$, and consider the vector field
\[
Y(\xi):=\frac{\nabla \tilde H(\xi)}{|\nabla \tilde H(\xi)|^{2}}, \qquad \xi\in U.
\]
Since $\tilde H\in C^{1,1}$ and $|\nabla \tilde H|\ge a/2$ on $U$, the field $Y$ is Lipschitz on $U$.
Let $\Phi(s,\theta)$ solve
\begin{equation}\label{eq:flowY}
\partial_{s}\Phi(s,\theta)=Y(\Phi(s,\theta)),
\qquad
\Phi(0,\theta)=z(\theta).
\end{equation}
The map $(s,\theta)\mapsto \Phi(s,\theta)$ gives a  bi-Lipschitz parametrization of an open neighborhood of $\Gamma_{h_0}$,  
with
\[
\|\Phi\|_{C^{0,1}}\le C(a,\|H\|_{C^{1,1}}),
\qquad
\|\Phi^{-1}\|_{C^{0,1}}\le C(a,\|H\|_{C^{1,1}}).
\]
As
\[
\frac{d}{ds}\tilde H(\Phi(s,\theta))
=\nabla \tilde H(\Phi(s,\theta))\cdot Y(\Phi(s,\theta))=1,
\] 
we have $\tilde H(\Phi(s,\theta))=h_{0}+s$ and  $\Phi(s,\cdot)$ gives a parametrization of $\Gamma_{h_0+s}$.

Define
\[
\gamma(s,\theta):=\Phi(s,\theta),
\qquad
J(s,\theta):=|\partial_{\theta}\gamma(s,\theta)|.
\]
Since $T_{\tilde H}$ is constant along each connected component of a level set, we have $T_{\tilde H}(\xi)=\tilde T(\tilde H(\xi))$ for $\xi\in U$, where $\tilde T$ is 
a one-variable period function given by the definition below:
\begin{equation}\label{eq:Ttilde-def}
\tilde T(h)
:=\int_{\Gamma_{h}}\frac{1}{|\nabla \tilde H|}\,d\mu
=\int_{0}^{\ell_{0}}\frac{J(h-h_{0},\theta)}{|\nabla \tilde H(\gamma(h-h_{0},\theta))|}\,d\theta.
\end{equation}
We now give a uniform Lipschitz bound for $\tilde T$  on $(h_{0}-\epsilon_{1},h_{0}+\epsilon_{1})$.

Set
\[
I(s,\theta):=\frac{J(s,\theta)}{|\nabla \tilde H(\gamma(s,\theta))|},
\]
we show that $I(\cdot,\theta)$ is uniformly Lipschitz in $s$.

\smallskip
\emph{(a) Control of $J$.}
By differentiating $\partial_{\theta}\gamma$ in $s$ and using $\partial_{s}\gamma=Y(\gamma)$, we have 
\[
\partial_{s}(\partial_{\theta}\gamma)=DY(\gamma)\,\partial_{\theta}\gamma.
\]
Hence for a.e.\ $s$,
\[
\partial_{s}J
=\partial_{s}|\partial_{\theta}\gamma|
\le \|DY\|_{L^{\infty}(U)}\,|\partial_{\theta}\gamma|
=\|DY\|_{L^{\infty}(U)}\,J.
\]
By Gr\"onwall's inequality, for $|s|<\epsilon_{1}$, 
\begin{equation}\label{eq:J-bound}
\begin{aligned}
    &J(s,\theta)\le e^{\|DY\|_{L^{\infty}(U)}|s|}\,J(0,\theta)\le e^{\epsilon_0\|DY\|_{L^{\infty}(U)}}\,|H|_{C^{1,1}}\leq  C\bigl(a,\|H\|_{C^{1,1}}\bigr).
\\&
|\partial_s J(s,\theta)|
\le \|DY\|_{L^{\infty}(U)}\,J(s,\theta)\leq \|DY\|_{L^{\infty}(U)} e^{\epsilon_0\|DY\|_{ L^{\infty}(U)}}\,|H|_{C^{1,1}}\leq  C\bigl(a,\| H\|_{C^{1,1}}\bigr).
\end{aligned}
\end{equation}

\smallskip
\emph{(b) Control of $\|\nabla \tilde H\|^{-1}$.}
Since $\partial_{s}\gamma=Y(\gamma)$ and $\nabla \tilde H$ is Lipschitz,
for a.e.\ $s$,
\[
\partial_{s}\bigl(\nabla \tilde H(\gamma(s,\theta))\bigr)
=D^{2}\tilde H(\gamma(s,\theta))\,\partial_{s}\gamma(s,\theta)
=D^{2}\tilde H(\gamma(s,\theta))\,Y(\gamma(s,\theta)).
\]
Using $|\nabla \tilde H|\ge a/2$ on $U$, we obtain
\begin{equation}\label{eq:invgrad-der}
\left|\partial_{s}\left(\frac1{|\nabla \tilde H(\gamma(s,\theta))|}\right)\right|
\le C\|D^{2}\tilde H\|_{L^{\infty}(U)}\,\|Y\|_{L^{\infty}(U)}
\le C\bigl(a,\| H\|_{C^{1,1}}\bigr).
\end{equation}

Combining \eqref{eq:J-bound}--\eqref{eq:invgrad-der} and again using $|\nabla \tilde H|\ge a/2$,
we find for a.e.\ $s$,
\[
|\partial_{s} I(s,\theta)|
\le \frac{|\partial_{s}J(s,\theta)|}{|\nabla \tilde H(\gamma(s,\theta))|}
+J(s,\theta)\left|\partial_{s}\left(\frac1{|\nabla \tilde H(\gamma(s,\theta))|}\right)\right|
\le C\bigl(a,\|H\|_{C^{1,1}}\bigr).
\]
Therefore, $I(\cdot,\theta)$ is Lipschitz in $s$ with a constant bounded by $C_{*}$ uniformly in $\theta$.
By \eqref{eq:Ttilde-def} and the dominated convergence theorem,
\[
|\tilde T(h)-\tilde T(h')|
\le \int_{0}^{\ell_{0}} |I(h-h_0,\theta)-I(h'-h_0,\theta)|\,d\theta
\le \ell_{0} C\bigl(a,\|H\|_{C^{1,1}}\bigr)|h-h'|.
\]
Together with the $L^{\infty}$ bound \eqref{eq:TLinfty-bound}, this yields
\[
\|\tilde T\|_{C^{0,1}(h_{0}-\epsilon_{1},h_{0}+\epsilon_{1})}
\le C\bigl(a,\|H\|_{C^{1,1}}\bigr),
\]
and  since $T_{\tilde H}(\xi)=\tilde T(\tilde H(\xi))$,
\[
\|T_{\tilde H}\|_{C^{0,1}(U)}
\le C\bigl(a,\|H\|_{C^{1,1}}\bigr)\,\|\tilde H\|_{C^{0,1}(U)}
\le C\bigl(a,\|H\|_{C^{1,1}}\bigr).
\]
Covering $\mathcal{R}_{2a}$ by finitely many such tubular neighborhoods $U$ completes the proof of (1).

The proof of (2) is similar, replacing Lipschitz estimates with estimates from $C^{k-1,\alpha}$ and 
differentiating \eqref{eq:Ttilde-def} in $h$ up to  order $k-1$.
\end{proof} 
\bibliographystyle{plain}
\bibliography{citation}

@article{Nadirashvili2013StationaryEuler,
  author  = {Nadirashvili, Nikolai},
  title   = {On Stationary Solutions of Two-Dimensional {Euler} Equation},
  journal = {Archive for Rational Mechanics and Analysis},
  volume  = {209},
  number  = {3},
  pages   = {729--745},
  year    = {2013},
  doi     = {10.1007/s00205-013-0642-8},
  url     = {https://doi.org/10.1007/s00205-013-0642-8}
}

@article{lundberg2021note,
  title={A note on the critical points of the localization landscape},
  author={Lundberg, Erik and Ramachandran, Koushik},
  journal={Complex Analysis and its Synergies},
  volume={7},
  number={2},
  pages={12},
  year={2021},
  publisher={Springer}
}

@article{huang2025rigidity,
  title={On the rigidity of uniformly rotating vortex patch near the {Rankine} vortex},
  author={Huang, Yupei},
  journal={Nonlinearity},
  volume={38},
  number={1},
  pages={015003},
  year={2025},
  publisher={IOP Publishing}
}

@article{brue2024enhanced,
  title={Enhanced dissipation for two-dimensional {Hamiltonian} flows},
  author={Bru{\`e}, Elia and Coti Zelati, Michele and Marconi, Elio},
  journal={Archive for Rational Mechanics and Analysis},
  volume={248},
  number={5},
  pages={84},
  year={2024},
  publisher={Springer}
}

@article{wolansky1998nonlinear,
  title={Nonlinear stability for saddle solutions of ideal flows and symmetry breaking},
  author={Wolansky, G and Ghil, M},
  journal={Communications in Mathematical Physics},
  volume={193},
  number={3},
  pages={713--736},
  year={1998},
  publisher={Springer}
}

@article{ball1984quasiconvexity,
  title={Quasiconvexity at the boundary, positivity of the second variation and elastic stability},
  author={Ball, JM and Marsden, JE},
  journal={Archive for Rational Mechanics and Analysis},
  volume={86},
  number={3},
  pages={251--277},
  year={1984},
  publisher={Springer New York}
}

@article{lin2004some,
  title={Some stability and instability criteria for ideal plane flows},
  author={Lin, Zhiwu},
  journal={Communications in Mathematical Physics},
  volume={246},
  number={1},
  pages={87--112},
  year={2004},
  publisher={Springer}
}

@article{KI,
  title={Characterization of steady solutions to the {2D Euler} equation},
  author={Izosimov, Anton and Khesin, Boris},
  journal={International Mathematics Research Notices},
  volume={2017},
  number={24},
  pages={7459--7503},
  year={2017},
  publisher={Oxford University Press}
}

@misc{sverak2011course,
  author       = {Šverák, Vladimír},
  title        = {Selected Topics in Fluid Mechanics (Course Notes)},
  year         = {2011/2012},
  note         = {Available online at \url{https://www-users.cse.umn.edu/~sverak/course-notes2011.pdf}},
  howpublished = {Lecture notes, University of Minnesota}
}

@article{nadirashvili1991wandering,
  title={Wandering solutions of {Euler's D-2 equation}},
  author={Nadirashvili, Nikolai Semenovich},
  journal={Functional Analysis and Its Applications},
  volume={25},
  number={3},
  pages={220--221},
  year={1991},
  publisher={Springer}
}

@book{arnold1998topological,
  title={Topological methods in hydrodynamics},
  author={Arnold, Vladimir I and Khesin, Boris A},
  year={1998},
  publisher={Springer}
}

@article{Kelvin,
  title={Maximum and minimum energy in vortex motion},
  author={William Thomson (Lord Kelvin)},
  journal={Nature},
  volume={22},
  number={574},
  pages={618--620},
  year={1880}
}

@article {Burton,
    AUTHOR = {Burton, G. R.},
     TITLE = {Rearrangements of functions, maximization of convex
              functionals, and vortex rings},
   JOURNAL = {Math. Ann.},
  FJOURNAL = {Mathematische Annalen},
    VOLUME = {276},
      YEAR = {1987},
    NUMBER = {2},
     PAGES = {225--253},
      ISSN = {0025-5831,1432-1807},
   MRCLASS = {49A50 (30C25 35J60 35R35)},
  MRNUMBER = {870963},
MRREVIEWER = {J.\ E.\ Rubio},
       DOI = {10.1007/BF01450739},
       URL = {https://doi.org/10.1007/BF01450739},
}

@article {GomezSerranoCompact,
    AUTHOR = {G\'omez-Serrano, Javier and Park, Jaemin and Shi, Jia},
     TITLE = {Existence of non-trivial non-concentrated compactly supported
              stationary solutions of the 2{D} {E}uler equation with finite
              energy},
   JOURNAL = {Mem. Amer. Math. Soc.},
  FJOURNAL = {Memoirs of the American Mathematical Society},
    VOLUME = {311},
      YEAR = {2025},
    NUMBER = {1577},
     PAGES = {v+82},
      ISSN = {0065-9266,1947-6221},
      ISBN = {978-1-4704-7530-7; 978-1-4704-8396-8},
   MRCLASS = {35Q35 (35Q31 35R35 76Dxx)},
  MRNUMBER = {4935841},
MRREVIEWER = {Qin\ Zhao},
       DOI = {10.1090/memo/1577},
       URL = {https://doi.org/10.1090/memo/1577},
}

@article{enciso2024smooth,
  title={Smooth nonradial stationary {Euler} flows on the plane with compact support},
  author={Enciso, Alberto and Fern{\'a}ndez, Antonio J and Ruiz, David},
  journal={arXiv preprint arXiv:2406.04414},
  year={2024}
}

@book{marchioro2012mathematical,
  title={Mathematical theory of incompressible nonviscous fluids},
  author={Marchioro, Carlo and Pulvirenti, Mario},
  volume={96},
  year={2012},
  publisher={Springer Science \& Business Media}
}

@article{robertson1936theory,
  title={On the theory of univalent functions},
  author={Robertson, Malcolm IS},
  journal={Annals of Mathematics},
  pages={374--408},
  year={1936},
  publisher={JSTOR}
}

@article{lin2011inviscid,
  title={Inviscid dynamical structures near {Couette flow}},
  author={Lin, Zhiwu and Zeng, Chongchun},
  journal={Archive for Rational Mechanics and Analysis},
  volume={200},
  number={3},
  pages={1075--1097},
  year={2011},
  publisher={Springer}
}

@article{coti2023stationary,
  title={Stationary structures near the {Kolmogorov and Poiseuille} flows in the {2D Euler} equations},
  author={Coti Zelati, Michele and Elgindi, Tarek M and Widmayer, Klaus},
  journal={Archive for Rational Mechanics and Analysis},
  volume={247},
  number={1},
  pages={12},
  year={2023},
  publisher={Springer}
}

@article{zhao2024inviscid,
  title={On the inviscid instability of the {2-D Taylor--Green vortex}},
  author={Zhao, Xinyu and Protas, Bartosz and Shvydkoy, Roman},
  journal={Journal of Fluid Mechanics},
  volume={999},
  pages={A64},
  year={2024},
  publisher={Cambridge University Press}
}

@article{cao2026instability,
  title={Instability of Two-Dimensional {Taylor--Green} Vortices},
  author={Cao-Labora, Gonzalo and Colombo, Maria and Dolce, Michele and Ventura, Paolo},
  journal={arXiv preprint arXiv:2601.23040},
  year={2026}
}

@article{ElgindiHuangSaidXie_ClassificationSteadyEulerFlows_DMJ,
  title   = {A Classification Theorem for {steady Euler flows}},
  author  = {Elgindi, Tarek and Huang, Yupei and Said, Ayman and Xie, Chunjing},
  journal = {Duke Mathematical Journal},
  year    = {2026},
  note    = {To appear}
}

@article{gui2024classification,
  title={On a classification of steady solutions to two-dimensional {Euler} equations},
  author={Gui, Changfeng and Xie, Chunjing and Xu, Huan},
  journal={arXiv preprint arXiv:2405.15327},
  year={2024}
}

@article{drivas2024geometric,
  title={A geometric characterization of steady laminar flow},
  author={Drivas, Theodore D and Nualart, Marc},
  journal={arXiv preprint arXiv:2410.18946},
  year={2024}
}

@article{elgindi2022regular,
  AUTHOR = {Elgindi, Tarek and Huang, Yupei},
TITLE = {Regular and singular steady states of the 2{D} incompressible
{E}uler equations near the {B}ahouri-{C}hemin patch},
JOURNAL = {Arch. Ration. Mech. Anal.},
FJOURNAL = {Archive for Rational Mechanics and Analysis},
VOLUME = {249},
YEAR = {2025},
NUMBER = {1},
PAGES = {Paper No. 2},
ISSN = {0003-9527},
MRCLASS = {35Q31},
MRNUMBER = {4841732},
MRREVIEWER = {Guodong Wang},
DOI = {10.1007/s00205-024-02077-6},
URL = {https://doi.org/10.1007/s00205-024-02077-6},
}

@article{lin2003instability,
  title={Instability of some ideal plane flows},
  author={Lin, Zhiwu},
  journal={SIAM Journal on Mathematical Analysis},
  volume={35},
  number={2},
  pages={318--356},
  year={2003},
  publisher={SIAM}
}

@article{lin2004nonlinear,
  title={Nonlinear instability of ideal plane flows},
  author={Lin, Zhiwu},
  journal={International Mathematics Research Notices},
  volume={2004},
  number={41},
  pages={2147--2178},
  year={2004},
  publisher={OUP}
}

@inproceedings{arnold1965conditions,
  title={Conditions for nonlinear stability of steady plane curvilinear flows of an ideal fluid},
  author={Arnold, VI},
  booktitle={Dokl. Akad. Nauk SSSR},
  volume={162},
  number={5},
  pages={975--978},
  year={1965}
}

@article{arnold1966priori,
  title={An a priori estimate in the theory of hydrodynamic stability},
  author={Arnold, Vladimir Igorevich},
  journal={Izv. Vyssh. Uchebn. Zaved. Mat.[Sov. Math. J.]},
  volume={5},
  pages={3},
  year={1966}
}

@article{choffrut2012local,
  AUTHOR = {Choffrut, Antoine and \v{S}ver\'{a}k, Vladim\'{\i}r},
TITLE = {Local structure of the set of steady-state solutions to the
2{D} incompressible {E}uler equations},
JOURNAL = {Geom. Funct. Anal.},
FJOURNAL = {Geometric and Functional Analysis},
VOLUME = {22},
YEAR = {2012},
NUMBER = {1},
PAGES = {136--201},
ISSN = {1016-443X},
MRCLASS = {35Q31 (37C25 37K65 46T05 58D05 76B15)},
MRNUMBER = {2899685},
MRREVIEWER = {Daniel Belti\c{t}\u{a}},
DOI = {10.1007/s00039-012-0149-8},
URL = {https://doi.org/10.1007/s00039-012-0149-8},
}

@article{constantin2021flexibility,
  AUTHOR = {Constantin, Peter and Drivas, Theodore and Ginsberg,
 Daniel},
 TITLE = {Flexibility and rigidity in steady fluid motion},
 JOURNAL = {Comm. Math. Phys.},
 FJOURNAL = {Communications in Mathematical Physics},
 VOLUME = {385},
 YEAR = {2021},
 NUMBER = {1},
 PAGES = {521--563},
 ISSN = {0010-3616},
 MRCLASS = {76E05 (35Q35)},
 MRNUMBER = {4275792},
 MRREVIEWER = {Hao Jia},
 DOI = {10.1007/s00220-021-04048-4},
 URL = {https://doi.org/10.1007/s00220-021-04048-4},
}

@article{hamel2017shear,
 AUTHOR = {Hamel, Fran\c{c}ois and Nadirashvili, Nikolai},
TITLE = {Shear flows of an ideal fluid and elliptic equations in
unbounded domains},
JOURNAL = {Comm. Pure Appl. Math.},
FJOURNAL = {Communications on Pure and Applied Mathematics},
VOLUME = {70},
YEAR = {2017},
NUMBER = {3},
PAGES = {590--608},
ISSN = {0010-3640},
MRCLASS = {35Q35},
MRNUMBER = {3602531},
MRREVIEWER = {Maria Specovius-Neugebauer},
DOI = {10.1002/cpa.21670},
URL = {https://doi.org/10.1002/cpa.21670},
}

@article{hamel2019liouville,
 AUTHOR = {Hamel, Fran\c{c}ois and Nadirashvili, Nikolai},
 TITLE = {A {L}iouville theorem for the {E}uler equations in the plane},
 JOURNAL = {Arch. Ration. Mech. Anal.},
 FJOURNAL = {Archive for Rational Mechanics and Analysis},
 VOLUME = {233},
 YEAR = {2019},
 NUMBER = {2},
 PAGES = {599--642},
 ISSN = {0003-9527},
 MRCLASS = {76B03 (35B53 35Q31)},
 MRNUMBER = {3951689},
 MRREVIEWER = {Anna L. Mazzucato},
 DOI = {10.1007/s00205-019-01364-x},
 URL = {https://doi.org/10.1007/s00205-019-01364-x},
}

@article{gomez2021symmetry,
  AUTHOR = {G\'{o}mez-Serrano, Javier and Park, Jaemin and Shi, Jia and Yao,
Yao},
TITLE = {Symmetry in stationary and uniformly rotating solutions of
active scalar equations},
JOURNAL = {Duke Math. J.},
FJOURNAL = {Duke Mathematical Journal},
VOLUME = {170},
YEAR = {2021},
NUMBER = {13},
PAGES = {2957--3038},
ISSN = {0012-7094},
MRCLASS = {35B06 (35Q35)},
MRNUMBER = {4312192},
DOI = {10.1215/00127094-2021-0002},
URL = {https://doi.org/10.1215/00127094-2021-0002},
}

@article{ruiz2023symmetry,
AUTHOR = {Ruiz, David},
TITLE = {Symmetry results for compactly supported steady solutions of
the 2{D} {E}uler equations},
JOURNAL = {Arch. Ration. Mech. Anal.},
FJOURNAL = {Archive for Rational Mechanics and Analysis},
VOLUME = {247},
YEAR = {2023},
NUMBER = {3},
PAGES = {Paper No. 40},
ISSN = {0003-9527},
MRCLASS = {35Q31 (76B03)},
MRNUMBER = {4575796},
MRREVIEWER = {Mikhail M. Shvartsman},
DOI = {10.1007/s00205-023-01877-6},
URL = {https://doi.org/10.1007/s00205-023-01877-6},
}

@article{drivas2023singularity,
AUTHOR = {Drivas, Theodore and Elgindi, Tarek},
TITLE = {Singularity formation in the incompressible {E}uler equation
in finite and infinite time},
JOURNAL = {EMS Surv. Math. Sci.},
FJOURNAL = {EMS Surveys in Mathematical Sciences},
VOLUME = {10},
YEAR = {2023},
NUMBER = {1},
PAGES = {1--100},
ISSN = {2308-2151},
MRCLASS = {35Q31 (35Q30)},
MRNUMBER = {4667415},
DOI = {10.4171/emss/66},
URL = {https://doi.org/10.4171/emss/66},
}

@article{yudovich2003eleven,
  title={Eleven great problems of mathematical hydrodynamics},
  author={Yudovich, Victor Iosifovich},
  journal={Moscow Mathematical Journal},
  volume={3},
  number={2},
  pages={711--737},
  year={2003},
  publisher={Независимый Московский университет--МЦНМО}
}
\end{document}